\documentclass[11pt]{article}
\usepackage[margin=1in]{geometry}
\usepackage{amsmath,amssymb,amsthm,mathtools}
\usepackage{enumitem}
\usepackage{microtype}
\usepackage{xcolor}
\usepackage{hyperref}
\hypersetup{colorlinks=true,linkcolor=blue!55!black,citecolor=blue!55!black,urlcolor=blue!55!black}
\numberwithin{equation}{section}
\newtheorem{theorem}{Theorem}[section]
\newtheorem{lemma}[theorem]{Lemma}
\newtheorem{proposition}[theorem]{Proposition}
\newtheorem{corollary}[theorem]{Corollary}
\theoremstyle{definition}

\newtheorem{assumption}[theorem]{Assumption}
\newtheorem{example}[theorem]{Example}
\theoremstyle{remark}
\newtheorem{remark}[theorem]{Remark}
\newcommand{\tr}{\operatorname{tr}}
\newcommand{\Span}{\operatorname{span}}

\newcommand{\Id}{I}
\newcommand{\fro}{\mathbin{:}}
\newcommand{\abs}[1]{\lvert #1\rvert}
\newcommand{\norm}[1]{\lVert #1\rVert}
\newcommand{\ip}[2]{\langle #1,#2\rangle}

\title{Geometric and stochastic framework for a regularized $p(x)$-Laplacian}
\author{Mustafa Avci\\Faculty of Science and Technology, Applied Mathematics, Athabasca University, Canada\\
\texttt{mavci@athabascau.ca (primary); avcixmustafa@gmail.com}}

\begin{document}
\maketitle

\begin{abstract}
For a prescribed classical profile and a fixed positive regularization parameter, we construct an It\^o diffusion whose interior generator reproduces the regularized $p(x)$-Laplacian when applied to that profile. We determine the exact tangent spectra of the unregularized and regularized fluxes, characterize their behavior at zero gradient, and establish quantitative ellipticity bounds on bounded gradient ranges. The full divergence expansion identifies a covariance equal to twice the gradient derivative of the regularized flux, evaluated at the profile gradient, and a logarithmic drift generated by spatial variation of the variable exponent  $p(x)$. Under explicit classical regularity assumptions, we establish boundedness and local Lipschitz continuity of the coefficients and prove strong existence and pathwise uniqueness up to domain exit. We identify the interior generator of the boundary-stopped process and verify the profile-reproduction identity. Further consequences include initial-point stability for a fixed global extension, localized martingale identities, and Dynkin's formula at bounded stopping times. An explicit nonconstant exponent example illustrates the construction at a critical point, and the coefficient selection is shown to be canonical under the flux-based prescriptions rather than determined uniquely by profile reproduction alone.
\end{abstract}



\section{Introduction}\label{sec:intro}

The problem addressed here is to construct, for a prescribed classical profile $u$ and a fixed $\varepsilon>0$, a strongly well-posed It\^o diffusion whose interior generator reproduces the regularized $p(x)$-Laplacian when applied to $u$. The  regularized $p(x)$-Laplacian operator under consideration is of the form
\begin{equation}\label{eq:intro-main-operator}
 \Delta_{p(\cdot),\varepsilon}v:=\mathrm{div}\left[(\varepsilon^2+\abs{\nabla v}^2)^{(p(x)-2)/2}\nabla v\right]\,\text{ in }\Omega.
\end{equation}

Variable exponent elliptic operators arise in models whose constitutive growth varies with spatial position. Applications include fluid mechanics \cite{ruzicka2007electrorheological,zhikov1997meyer}, flow through porous media \cite{amaziane2009nonlinear}, magnetostatics \cite{cekic2012lp}, image restoration \cite{chen2006variable}, and capillarity \cite{avci2013ni}; while related anisotropic systems are considered in \cite{boureanu2023}. The analysis of these models rests on the theory of variable exponent Lebesgue and Sobolev spaces. Their structural properties were developed in \cite{FanZhao,KovacikRakosnik} and are treated systematically in the monographs \cite{cruz2013variable,diening2011lebesgue,RadulescuRepovs}, which provide the functional-analytic and variational framework for equations with spatially dependent growth.

Within this framework, regularity and the relation between different solution concepts have received substantial attention. Acerbi and Mingione \cite{AcerbiMingione} establish gradient estimates for
$p(x)$-Laplace systems, while Fan \cite{FanGlobal} develops global $C^{1,\alpha}$ regularity for variable-exponent equations in divergence
form. Juutinen, Lukkari, and Parviainen \cite{JuutinenLukkariParviainen} prove equivalence between weak and viscosity solutions of the homogeneous divergence-form $p(x)$-Laplace
equation; Medina and Ochoa \cite{MedinaOchoa} address nonhomogeneous equations under suitable hypotheses. For the distinct normalized operator, Siltakoski \cite{Siltakoski} relates viscosity solutions to distributional weak solutions of the associated strong $p(x)$-Laplace equation. Of particular relevance to the present construction, the analysis in \cite{JuutinenLukkariParviainen} also
identifies the logarithmic term generated by spatial variation of the exponent and employs the tangent matrix and estimates for its square root in the comparison argument.

Probabilistic approaches provide a complementary perspective. Peres and Sheffield \cite{PeresSheffield} connect the constant exponent $p$-Laplace equation with tug-of-war games with noise, and Manfredi,
Parviainen, and Rossi \cite{ManfrediParviainenRossi} characterize $p$-harmonic functions through asymptotic mean-value formulas. Arroyo, Heino, and Parviainen \cite{ArroyoHeinoParviainen} extend the
game framework to spatially varying probabilities associated with the normalized $p(x)$-Laplacian. Mean-value descriptions also encompass the variational $p$-Laplacian \cite{DelTesoLindgren} and, for $p>2$, nonhomogeneous equations with forcing \cite{DelTesoRossi}. In \cite{avci2026un}, Avci established existence and uniqueness results for a time-inhomogeneous nonlinear SDE whose drift and diffusion are governed by state-dependent variable exponents. This framework generalizes geometric Brownian motion and constant elasticity of variance (CEV) models. It also provides a suitable setting in which the deterministic Poisson equation admits a probabilistic representation via a time-homogeneous nonlinear SDE with state-dependent variable exponents.

A more direct precursor to the present diffusion construction is the paper of Luiro and Parviainen
\cite{LuiroParviainenPreprint}: for a prescribed constant exponent $p$-harmonic function, they construct a continuous-time diffusion and obtain an exact exit representation under a restriction on
the critical set. Away from critical points, its second-order coefficient shares the radial--tangential structure of the unregularized flux derivative, with a normalized scalar prefactor. These connections motivate a construction for the full variable exponent divergence operator that retains the exponent-gradient contribution and controls the coefficients through zero gradients.

The specific gain of the present paper is the explicit assembly and quantitative control of these ingredients for the full regularized variable exponent operator at an arbitrary prescribed classical profile. The covariance retains the flux's unnormalized tangent scale, the drift retains its explicit spatial derivative, and fixed positive regularization allows coefficient estimates through all zero gradients without a restriction on the critical set. These estimates lead to a pathwise-unique stopped strong diffusion and its exact profile-reproduction identity. The spectral calculations and Lipschitz SDE methods themselves are standard; the contribution is this coefficient-level framework and its precise hypotheses, bounds, and stochastic consequences. Two difficulties must be resolved: the unregularized tangent matrix can degenerate or become singular at zero gradient, and spatial variation of $p$ produces the logarithmic first-order contribution $c_\varepsilon(x,\nabla u)\cdot\nabla u$ in \eqref{eq:intro-expansion}, whose coefficient contains $\nabla p$. For the unregularized flux
\begin{equation}\label{eq:unreg-flux-intro}
A(x,\xi)=\abs{\xi}^{p(x)-2}\xi,\quad A(x,0)=0,
\end{equation}
the tangent matrix $D(x,\xi)=D_\xi A(x,\xi)$ tends to zero as $\xi\to0$ when $p(x)>2$ and has unbounded eigenvalues when $1<p(x)<2$, with $x$ fixed. These difficulties motivate regularization and the complete spatial expansion in the proposed stochastic construction.

We prescribe a bounded domain $\Omega\subset\mathbb{R}^N$, $N\geq1$, an open set $\mathcal{U}$ containing $\overline\Omega$ compactly, and data $p\in C^{1,1}(\mathcal{U})$, $u\in C^{2,1}(\mathcal{U})$; see Assumption \ref{ass:classical-framework}. Our objective is to obtain strong existence and pathwise uniqueness up to domain exit and to prove the generator's exact action on this prescribed profile.
We first determine the flux geometry and regularize it by setting
\begin{align}
 r_\varepsilon(\xi)&=(\varepsilon^2+\abs{\xi}^2)^{1/2},\label{eq:intro-r-eps}\\
 A_\varepsilon(x,\xi)&=r_\varepsilon(\xi)^{p(x)-2}\xi,\label{regularized-flux}\\
 D_\varepsilon(x,\xi)&=D_\xi A_\varepsilon(x,\xi),\label{eq:intro-reg-tangent}\\
 c_\varepsilon(x,\xi)&=r_\varepsilon(\xi)^{p(x)-2}\log r_\varepsilon(\xi)\,\nabla p(x).
 \label{eq:intro-c-eps}
\end{align}
The classical chain rule gives
\begin{equation}\label{eq:intro-expansion}
 \mathrm{div} A_\varepsilon(x,\nabla u)=D_\varepsilon(x,\nabla u)\fro D^2u+c_\varepsilon(x,\nabla u)\cdot\nabla u.
\end{equation}
We then freeze the gradient at $\nabla u(x)$ and define
\begin{align}
 b_{u,\varepsilon}(x)&=c_\varepsilon(x,\nabla u(x)),\label{eq:intro-frozen-coefficients}\\
 D_{u,\varepsilon}(x)&=D_\varepsilon(x,\nabla u(x)),\label{eq:intro-frozen-coefficients-a}\\
 \sigma_{u,\varepsilon}(x)&=(2D_{u,\varepsilon}(x))^{1/2},\label{eq:intro-sigma}
\end{align}
where the square root is the principal symmetric positive-definite square root. The covariance is $2D_{u,\varepsilon}$; its factor $2$ compensates for the factor $1/2$ in It\^o's formula. With an $N$-dimensional Brownian motion $W_t$ and exit time $\tau_\Omega=\inf\{t\geq0:X_t\notin\Omega\}$, the process to be constructed solves
\begin{equation}\label{eq:intro-profile-sde}
\left\{\begin{aligned}
dX_t&=b_{u,\varepsilon}(X_t)dt+\sigma_{u,\varepsilon}(X_t)dW_t,&&0\leq t<\tau_\Omega,\\X_0&=x\in\Omega,
\end{aligned}\right.
\end{equation}
Our method combines exact spectral and chain-rule calculations with matrix square-root estimates, cutoff localization, Picard iteration, and stopped It\^o formulas
\cite{Bhatia1997,HighamFunctions,KaratzasShreve,Oksendal,StroockVaradhan}.

We obtain the exact radial--tangential spectra, the dimension-one corrections, and quantitative ellipticity bounds for fixed $\varepsilon>0$ on bounded gradient ranges. In particular,
$D_\varepsilon(x,0)=\varepsilon^{p(x)-2}\Id$, so the regularized matrix remains positive definite at critical points. The frozen coefficients are bounded on $\overline\Omega$ and locally Lipschitz on $\mathcal{U}$. Theorem \ref{thm:sde-existence} consequently gives a pathwise-unique strong solution of \eqref{eq:intro-profile-sde} up to exit, with no explosion before $\tau_\Omega$. Theorem \ref{thm:companion-generator} identifies the interior generator of the boundary-stopped process as
\begin{equation}\label{eq:intro-generator}
 \mathcal{L}_{u,\varepsilon}\phi=D_{u,\varepsilon}\fro D^2\phi+b_{u,\varepsilon}\cdot\nabla\phi,
\end{equation}
and proves
\begin{equation}\label{eq:intro-profile-reproduction}
\mathcal{L}_{u,\varepsilon}u=\Delta_{p(\cdot),\varepsilon}u=\mathrm{div} A_\varepsilon(x,\nabla u)\,\text{ in }\Omega.
\end{equation}
Further consequences are initial-point stability for a fixed global extension, localized martingale identities, and Dynkin's formula at bounded stopping times
(Corollaries \ref{cor:sde-stability}, \ref{cor:local-martingale}, and \ref{cor:dynkin}).

The construction gives an explicit connection between the tangent geometry of a nonlinear variable exponent operator and a classical stopped diffusion. Its covariance records the directional response at
$\nabla u$, while its drift retains the contribution of $\nabla p$. Once $u$ is fixed, the generator is linear in its test function, making stochastic calculus available for identities involving that profile. For example, if $-\Delta_{p(\cdot),\varepsilon}u=f$, then $\mathcal{L}_{u,\varepsilon}u=-f$; the eigenprofile consequence in Corollary \ref{cor:generator-eigenprofile} has the same conditional
character. The profile remains prescribed: the construction supplies neither a PDE existence theorem nor a stochastic evolution of $u$. The companion generator also generally differs from the Fr\'echet linearization (Remark~\ref{rem:frechet-linearization}) $v\mapsto \mathrm{div}(D_{u,\varepsilon}\nabla v)$. All stochastic conclusions concern fixed $\varepsilon>0$; the coefficient limits away from zero gradient do not establish convergence of the diffusions across critical points as $\varepsilon\downarrow0$.

The organization of the paper is as follows. Section \ref{sec:framework} states the deterministic and stochastic assumptions and fixes the notation. Section \ref{sec:unreg-geometry} develops the unregularized tangent geometry and divergence expansion. Section \ref{sec:regularization} establishes the regularized geometry, quantitative ellipticity, and regularity of the frozen coefficient fields. Section \ref{sec:main-results} constructs the diffusion and proves its generator and stochastic identities.  Appendix \ref{app:example} supplies an nonconstant exponent example for illustration of the construction.

\section{Analytic, deterministic, and stochastic framework}\label{sec:framework}

We distinguish the hypotheses needed for pointwise flux geometry, for classical divergence expansions, and for the stochastic construction.

\subsection{Domain, exponent, and prescribed profile}

Throughout, $N\geq1$, $\Omega\subset\mathbb{R}^N$ is a nonempty open domain, and $\varepsilon>0$ is fixed whenever a regularized quantity is used. For sets $K$ and an open set $G$, the notation $K\Subset G$ means that $\overline K$ is compact and contained in $G$.
\begin{assumption}\label{ass:exponent-range}
Let $p:\Omega\to(1,\infty)$ be a measurable function. Fix constants $p^-$ and $p^+$ such that
\begin{equation}\label{eq:p-bounds}
1<p^-\leq p(x)\leq p^+<\infty\quad(x\in\Omega).
\end{equation}
Here $p^-$ and $p^+$ denote uniform lower and upper bounds; they are not required to equal the essential extrema on $\Omega$. 
\end{assumption}
For the spectrum at a fixed pair $(x,\xi)$, only the number $p(x)>1$ is needed; continuity of $p$ is unnecessary. The common bounds in \eqref{eq:p-bounds} are used when estimates must hold uniformly in $x$.
\begin{assumption}\label{ass:classical-deterministic}
Assumption \ref{ass:exponent-range} holds, $p\in C^1(\Omega)$, and the prescribed profile satisfies $u\in C^2(\Omega)$. For the regularized operator, $\varepsilon>0$ is fixed.
\end{assumption}

These hypotheses justify the classical chain rule. The unregularized expansion is used on the noncritical set $\{x\in\Omega:\nabla u(x)\neq0\}$; the regularized expansion is valid throughout $\Omega$.

\begin{assumption}\label{ass:classical-framework}\label{ass:classical-profile}\label{ass:framework} The domain $\Omega$ is bounded and there is an open set $\mathcal{U}$ such that
$\overline\Omega\Subset\mathcal{U}$. For a fixed $\varepsilon>0$, the exponent and the prescribed deterministic profile satisfy
\begin{equation}\label{eq:p-range}
 \begin{gathered}
p\in C^{1,1}(\mathcal{U}),\quad u\in C^{2,1}(\mathcal{U}),\\
1<p^-\leq p(x)\leq p^+<\infty\quad x\in\mathcal{U}.
 \end{gathered}
\end{equation}
\end{assumption}
In this stronger framework, the same chosen bounds $p^-,p^+$ apply on $\mathcal U$; no equality with the essential extrema on $\Omega$ is imposed. Here $C^{k,1}$ means $C^k$ regularity with derivatives of order $k$ locally Lipschitz. Fix once and for all an open set $V$ satisfying
\begin{equation}\label{eq:framework-neighborhood}
\overline\Omega\subset V\Subset\mathcal{U}, \quad M_u:=\sup_{\overline V}\abs{\nabla u}<\infty.
\end{equation}
The compact neighborhood used for the profile estimates is
$\overline V$, and the corresponding compact set for functions of both
position and gradient is
\begin{equation}\label{eq:framework-product-set}
K:=\overline V\times\{\xi\in\mathbb{R}^N:|\xi|\leq M_u\}.
\end{equation}
If $M_u=0$, the gradient factor in $K$ is the singleton $\{0\}$. Bounds for $p$, $\nabla p$, $\nabla u$, and $D^2u$ are taken on $\overline V$; bounds for $A_\varepsilon$, $D_\varepsilon$, and their derivatives in the gradient variable are taken on $K$.  Any additional neighborhood used for a cutoff or a local estimate will be named explicitly.

The first derivatives of $p$ and the second derivatives of $u$ are also bounded on $\overline V$. This neighborhood formulation provides coefficients up to the exit point and permits extensions used in the SDE construction.  It does not require smoothness of $\partial\Omega$, a boundary condition on $u$, or positivity of $u$.

Assumption \ref{ass:classical-framework} is a sufficient regularity framework for the stochastic results. It is imposed on the data, rather than derived from an elliptic regularity theorem. In particular, the condition $p\in C^{1,1}$ controls the local Lipschitz regularity of $\nabla p$ in the drift. Every use of this stronger framework will be indicated explicitly.

\subsection{Flux, energy density, and differentiation conventions}

We use $A(x,0)=0$ for the continuous extension of the unregularized flux in \eqref{eq:unreg-flux-intro}. Its gradient derivative $D(x,\xi)$ is initially considered for $\xi\neq 0$. The regularized flux $A_\varepsilon$ and its derivative $D_\varepsilon$ are defined for all $\xi\in\mathbb{R}^N$. The unregularized and regularized energy densities are defined by
\begin{equation}\label{eq:framework-energy-densities}
\Phi(x,\xi)=\frac{\abs{\xi}^{p(x)}}{p(x)}\quad \text{and} \quad \Phi_\varepsilon(x,\xi)=\frac{(\varepsilon^2+\abs{\xi}^2)^{p(x)/2}-\varepsilon^{p(x)}}{p(x)},
\end{equation}
respectively. At fixed $x$, differentiation of $\Phi$ and $\Phi_\varepsilon$  in $\xi$ gives
\begin{equation}\label{eq:framework-energy-derivatives}
 D_\xi\Phi=A,\quad D_\xi\Phi_\varepsilon=A_\varepsilon,\quad
 D_\xi^2\Phi_\varepsilon=D_\varepsilon A_\varepsilon=D_\varepsilon.
\end{equation}
For $\xi\neq0$, $D_\xi A(x,\xi)=D_\xi^2\Phi(x,\xi)$ is the Hessian of the unregularized energy. For every $\xi$, $D_\xi A_\varepsilon(x,\xi)=D_\xi^2\Phi_\varepsilon(x,\xi)$ is the Hessian of the regularized energy. Notice that the subtraction of $\varepsilon^{p(x)}/p(x)$ sets $\Phi_\varepsilon(x,0)=0$ and does not affect these gradient derivatives.  We use the energy densities to explain the tangent matrix; selecting a profile by minimizing an energy is outside the construction studied here.

The notation $D_\xi$ always denotes differentiation in the gradient variable while holding $x$ fixed. Conversely, an explicit spatial derivative of $A_\varepsilon(x,\xi)$ holds $\xi$ fixed. For a composed flux, these are distinct parts of the chain rule
\begin{align}\label{eq:framework-chain-rule}
\mathrm{div} A_\varepsilon(x,\nabla u(x)) =\sum_i(\partial_{x_i}A_{\varepsilon,i})(x,\nabla u(x))+\sum_{i,j}(\partial_{\xi_j}A_{\varepsilon,i})(x,\nabla u(x))u_{x_jx_i}(x).
\end{align}
This convention explains the two coefficients in \eqref{eq:intro-expansion}.  At a critical point, the scalar contribution $c_\varepsilon(x,\nabla u)\cdot\nabla u$ is zero, although the vector $c_\varepsilon(x,0)=\varepsilon^{p(x)-2}\log\varepsilon\,\nabla p(x)$ need not be zero.

\subsection{Analytic and matrix notation}

For a scalar function $v$, we write
\begin{equation}\label{eq:gradient-hessian-notation}
\nabla v=(v_{x_1},\ldots,v_{x_N})^\top,\quad D^2v=(v_{x_ix_j})_{i,j=1}^N,\quad \Delta v=\tr(D^2v).
\end{equation}
The symbols $\abs{\xi}$ and $\ip{\xi}{\eta}$ denote the Euclidean norm and inner product. For vectors $\xi,\eta\in\mathbb{R}^N$,
\begin{equation}\label{eq:tensor-definition}
\xi\otimes\eta=(\xi_i\eta_j)_{i,j=1}^N.
\end{equation}
For $\xi\neq0$, the matrix
\begin{equation}\label{eq:projection-definition}
 P_\xi=\frac{\xi\otimes\xi}{\abs{\xi}^2}
\end{equation}
is the orthogonal projection onto $\Span\{\xi\}$.  Hence
\begin{equation}\label{eq:parallel-perpendicular-decomposition}
\eta=\eta_\parallel+\eta_\perp,\quad
\eta_\parallel=P_\xi\eta,\quad
 \eta_\perp=(\Id-P_\xi)\eta.
\end{equation}
If $N=1$, then $P_\xi=\Id$ and $\eta_\perp=0$.  No projection $P_0$ is assigned; formulas at zero gradient are evaluated directly from the regularized matrix.

For matrices $M,H\in\mathbb{R}^{N\times N}$, the Frobenius contraction is
\[
 M\fro H=\sum_{i,j=1}^N M_{ij}H_{ij}=\tr(MH^\top).
\]
We use the operator norm and Frobenius norm
\begin{align}
 \norm{M}_{\mathrm{op}}&=\sup_{\abs{\eta}=1}\abs{M\eta},\label{eq:operator-matrix-norm}\\
\norm{M}_{\mathrm F}&=(M\fro M)^{1/2}=\left(\sum_{i,j=1}^N\abs{M_{ij}}^2\right)^{1/2}.\label{eq:frobenius-matrix-norm}
\end{align}
They satisfy
\begin{equation}\label{eq:matrix-norm-equivalence}
\norm{M}_{\mathrm{op}}\leq\norm{M}_{\mathrm F}\leq\sqrt N\,\norm{M}_{\mathrm{op}}.
\end{equation}
For a differentiable matrix field, our divergence convention is
\begin{equation}\label{eq:framework-matrix-divergence}
(\mathrm{div} M)_j=\sum_{i=1}^N\partial_{x_i}M_{ij},\quad\mathrm{div}(M\nabla\phi)=M\fro D^2\phi+(\mathrm{div} M)\cdot\nabla\phi.
\end{equation}

\begin{remark}\label{rem:proof-matrix-norms-paper1}
An unsubscripted matrix norm in the deterministic coefficient and square-root estimates means $\norm{\cdot}_{\mathrm{op}}$.  In the SDE existence and stability proofs, every norm of a diffusion matrix,
including those defining its bounds and Lipschitz constants, means $\norm{\cdot}_{\mathrm F}$.  This convention ensures that the constants used in the It\^o isometry and martingale estimates are the ones actually defined for those proofs.  Formula~\eqref{eq:matrix-norm-equivalence} transfers earlier operator-norm bounds with the required factor $\sqrt N$.
\end{remark}

A symmetric matrix field $D$ is uniformly elliptic on $G$ if constants $0<\alpha\leq\beta<\infty$ satisfy
\begin{equation}\label{eq:uniform-ellipticity-definition}
\alpha\abs{\eta}^2\leq\eta^\top D(x)\eta\leq\beta\abs{\eta}^2 \quad x\in G,\ \eta\in\mathbb{R}^N.
\end{equation}
The constants are independent of $x$ and $\eta$ on the indicated set; in this paper they may depend on $\varepsilon$, the exponent bounds, and a bound for the profile gradient.  For a symmetric positive-definite matrix $S$, its spectral condition number is
\[
\kappa(S)=\norm{S}_{\mathrm{op}}\norm{S^{-1}}_{\mathrm{op}}=\frac{\lambda_{\max}(S)}{\lambda_{\min}(S)}.
\]
In particular, it equals one for a positive $1\times1$ matrix.

Let $\operatorname{Sym}_N^{++}$ be the cone of symmetric positive-definite matrices.  Every $S\in\operatorname{Sym}_N^{++}$ has a unique principal square root $S^{1/2}\in\operatorname{Sym}_N^{++}$.  The square-root map is smooth on this cone and Lipschitz on sets with a common positive lower spectral bound and a finite upper bound \cite{Bhatia1997,HighamFunctions}.  This is the matrix fact used to pass
from $D_{u,\varepsilon}$ to $\sigma_{u,\varepsilon}$.

\subsection{Probability space, exit times, generator and local strong solutions}

Let $(\Xi,\mathcal F,(\mathcal F_t)_{t\geq0},\mathbb{P})$ be a filtered probability space with a complete, right-continuous filtration, carrying an $N$-dimensional Brownian motion $W_t$ with respect to that filtration. The initial point $x$ is deterministic; $\mathbb{E}^x$ denotes expectation for the solution started at $x$.

For coefficient fields $b:G\to\mathbb{R}^N$ and $\sigma:G\to\mathbb{R}^{N\times N}$ on an open set $G$, consider
\begin{equation}\label{eq:generic-sde}
dX_t=b(X_t)dt+\sigma(X_t)dW_t,\quad X_0=x\in G.
\end{equation}
A strong solution is a continuous adapted process driven by the given Brownian motion that satisfies the corresponding integral equation up to its lifetime.  Pathwise uniqueness means that two such solutions on the same probability space, with the same Brownian motion and initial point, agree up to their common lifetime.

The covariance matrix and interior generator associated with \eqref{eq:generic-sde} are
\begin{equation}\label{eq:generic-generator}
a=\sigma\sigma^\top,\quad \mathcal{L}\phi=b\cdot\nabla\phi+\frac12a\fro D^2\phi.
\end{equation}
For the profile coefficients \eqref{eq:intro-frozen-coefficients}--\eqref{eq:intro-sigma}, the covariance matrix is
\begin{equation}\label{eq:profile-covariance}
a_{u,\varepsilon}:=\sigma_{u,\varepsilon}\sigma_{u,\varepsilon}^{\top}=2D_{u,\varepsilon}.
\end{equation}
Thus $\sigma_{u,\varepsilon}$ is the diffusion coefficient, whereas $a_{u,\varepsilon}$ is the covariance; they are distinct objects. Notice that formula \eqref{eq:generic-generator} gives exactly
\eqref{eq:intro-generator}.

The geometric estimates and chain-rule calculations in the following sections establish that these coefficients satisfy the hypotheses above and that their generator has the required action on the prescribed profile.

For a continuous adapted process and open $G$, the first exit time $\tau_G=\inf\{t\geq0:X_t\notin G\}$ is a stopping time, and if $X_0\in G$ then $X_{\tau_G}\in\partial G$ on $\{\tau_G<\infty\}$.  If $\theta$ is a bounded stopping time and the coefficients and derivatives of $\phi\in C^2$ are bounded on the stopped region, It\^o's formula gives  
\begin{align}\label{eq:ito-pathwise}
 \phi(X_\theta)&=\phi(x)+\int_0^\theta\nabla\phi(X_s)^\top\sigma(X_s)dW_s+\int_0^\theta\mathcal{L}\phi(X_s)ds,
\end{align}
from which one obtains Dynkin's formula
\begin{align}
 \mathbb{E}^x\phi(X_\theta)&=\phi(x)+\mathbb{E}^x\int_0^\theta\mathcal{L}\phi(X_s)ds.
\end{align}
The following standard local well-posedness result is recalled from \cite{KaratzasShreve,Oksendal,StroockVaradhan}.
\begin{proposition}\label{prop:sde-background}
Suppose $b$ and $\sigma$ are locally Lipschitz on $G$.  For every $x\in G$ there is a pathwise unique  strong solution of \eqref{eq:generic-sde}.  Its lifetime is the limit of the exit times from
an increasing exhaustion of $G$ by relatively compact open sets.  If $\Omega\Subset G$ and $x\in\Omega$, the solution is strongly well posed up to
\begin{equation}\label{eq:framework-exit-time}
 \tau_\Omega=\inf\{t\geq0:X_t\notin\Omega\}
\end{equation}
and cannot explode before that exit.  Its stopped path, including the finite exit point, can be constructed from a globally nonexplosive extension of the coefficients.
\end{proposition}

\section{Tangent diffusion geometry of the unregularized flux}\label{sec:unreg-geometry}

First, we start with the derivative and spectral decomposition.
\newpage
\begin{lemma}\label{lem:unreg-spectrum}
Fix $x\in\Omega$ with $p(x)>1$ and let $\xi\neq 0$.  No regularity of $p$
beyond its value at $x$ is required.  Then:
\begin{enumerate}[label=(\alph*)]
\item The flux derivative is the symmetric matrix
\begin{equation}\label{eq:D-unreg}
 D_{ij}(x,\xi)=\abs{\xi}^{p(x)-2}\delta_{ij}+(p(x)-2)\abs{\xi}^{p(x)-4}\xi_i\xi_j.
\end{equation}
Equivalently,
\begin{equation}\label{eq:D-unreg-projection}
 D(x,\xi)=\abs{\xi}^{p(x)-2}\left(\Id+(p(x)-2)P_\xi\right).
\end{equation}
\item On the radial subspace $\Span\{\xi\}$, $D(x,\xi)$ acts by the eigenvalue
\begin{equation}\label{eq:mu-parallel-unreg}
 \mu_\parallel(x,\xi)=(p(x)-1)\abs{\xi}^{p(x)-2}.
\end{equation}
\item For every $N\geq1$, define the scalar
\begin{equation}\label{eq:mu-perp-unreg}
 \mu_\perp(x,\xi)=\abs{\xi}^{p(x)-2}.
\end{equation}
If $N\geq2$, $D(x,\xi)$ acts on the $(N-1)$-dimensional subspace $\xi^\perp$ by this tangential eigenvalue. If $N=1$, the scalar is auxiliary: $\Id-P_\xi=0$, and no tangential eigenvalue is present.
\item The orthogonal decomposition gives
\begin{equation}\label{eq:D-spectral-unreg}
 D(x,\xi)=\mu_\parallel P_\xi+\mu_\perp(\Id-P_\xi).
\end{equation}
When $N=1$, only the radial eigenvalue is present.  If $N\geq2$ and $p(x)\neq 2$, the two displayed eigenvalues are distinct, with eigenspaces $\Span\{\xi\}$ and $\xi^\perp$ and multiplicities $1$ and $N-1$, respectively.  If $p(x)=2$, then $D(x,\xi)=\Id$; its sole eigenvalue is $1$, with multiplicity $N$ and eigenspace $\mathbb{R}^N$.
\end{enumerate}
\end{lemma}
\begin{proof}
$(a)$ Fix $x$. For $\xi\neq 0$,
\[
 \frac{\partial}{\partial\xi_j}\abs{\xi}^{p(x)-2}=(p(x)-2)\abs{\xi}^{p(x)-4}\xi_j.
\]
Thus, applying the product rule gives \eqref{eq:D-unreg}
\[
 \frac{\partial A_i}{\partial\xi_j}=\abs{\xi}^{p(x)-2}\delta_{ij}+(p(x)-2)\abs{\xi}^{p(x)-4}\xi_i\xi_j,
\]
where $A_i(\xi)=\abs{\xi}^{p(x)-2}\xi_i$ is $i$th component of the flux, and $\delta_{ij}$ is the Kronecker delta. On the other hand, factoring out $\abs{\xi}^{p(x)-2}$ and using the definition of $P_\xi$ gives the matrix representation \eqref{eq:D-unreg-projection}. For fixed $x$, the function $\Phi(x,\xi)=|\xi|^{p(x)}/p(x)$ is $C^2$ on $\{\xi\neq0\}$ and
$D_\xi A(x,\xi)=D_\xi^2\Phi(x,\xi)$; hence the matrix is symmetric. Indeed, since $I$ is symmetric and
\[
(\xi\otimes\xi)^T=\xi\otimes\xi
\]
due to the fact that
\[
(\xi\otimes\xi)_{ij}=\xi_i\xi_j=\xi_j\xi_i=(\xi\otimes\xi)_{ji},
\]
it follows that
\[
D_\xi A(x,\xi)^T=D_\xi A(x,\xi).
\]

$(b)$ For any vector $\eta \in \mathbb{R}^N$,
\[
P_\xi \eta=\frac{\xi\cdot \eta}{|\xi|^2}\xi.
\]
If $\eta=c\xi$, then $P_\xi \eta=\eta$. Notice that since $P_\xi \eta$ is a scalar multiple of $\xi$, $P_\xi \eta \in \Span\{\xi\}$.  Finally formula \eqref{eq:D-unreg-projection} gives
\[
D(x,\xi)\eta=|\xi|^{p(x)-2}[I+(p(x)-2)P_\xi]\eta=(p(x)-1)|\xi|^{p(x)-2}\eta.
\]

$(c)$ If $\eta\perp\xi$, then $P_\xi \eta=0$, and hence
\[
 D(x,\xi)\eta=\abs{\xi}^{p(x)-2}\eta.
\]

$(d)$ Every vector $\eta\in\mathbb{R}^N$ decomposes uniquely as $\eta=P_\xi\eta+(\Id-P_\xi)\eta$.  The computations in $(b)$ and $(c)$ show that these are orthogonal invariant subspaces, so
\[
 D(x,\xi)\eta=\mu_\parallel P_\xi\eta+\mu_\perp(\Id-P_\xi)\eta.
\]
This proves \eqref{eq:D-spectral-unreg}.  If $N\geq2$, both subspaces contain nonzero vectors and
\[
\operatorname{spec}\bigl(D(x,\xi)\bigr)=\left\{(p(x)-1)|\xi|^{p(x)-2},\,|\xi|^{p(x)-2}\right\}.
\]
When $p(x)\neq2$, the two numbers are distinct and have multiplicities $1$ and $N-1$, respectively.  When $p(x)=2$, they coincide: the spectrum is the singleton $\{1\}$ and its eigenvalue has multiplicity $N$, with eigenspace $\mathbb{R}^N$.  In dimension one, $\xi^\perp=\{0\}$ supplies no eigenvalue, and instead
\[
\operatorname{spec}\bigl(D(x,\xi)\bigr)=\left\{(p(x)-1)|\xi|^{p(x)-2}\right\}.
\]
\end{proof}

Next, we obtain some sharp ellipticity bounds.

\begin{lemma}\label{thm:unreg-ellipticity}
Fix $x\in\Omega$ with $p(x)>1$ and let $\xi\neq 0$.  Then, for every $\eta\in\mathbb{R}^N$,
\begin{equation}\label{eq:unreg-bounds}
 \min\{1,p(x)-1\}\abs{\xi}^{p(x)-2}\abs{\eta}^2\leq \eta^\top D(x,\xi)\eta\leq \max\{1,p(x)-1\}\abs{\xi}^{p(x)-2}\abs{\eta}^2.
\end{equation}
If $N\geq2$, both constants are sharp and
\begin{equation}\label{eq:kappa-unreg}
 \kappa(D(x,\xi))=\max\left\{p(x)-1,\frac1{p(x)-1}\right\}.
\end{equation}
If $N=1$, then
\begin{equation}\label{eq:N1-D}
 D(x,\xi)=(p(x)-1)\abs{\xi}^{p(x)-2}
\end{equation}
as a $1\times1$ matrix, and
\begin{equation}\label{eq:kappa-N1}
 \kappa(D(x,\xi))=1.
\end{equation}
\end{lemma}
\begin{proof}
$(a)$ Write
$\eta=\eta_\parallel+\eta_\perp$ with
$\eta_\parallel=P_\xi\eta$ and
$\eta_\perp=(\Id-P_\xi)\eta$.  By orthogonality and Lemma~\ref{lem:unreg-spectrum},
\begin{equation}\label{eq:unreg-quadratic-exact}
\eta^\top D(x,\xi)\eta=\abs{\xi}^{p(x)-2}\left((p(x)-1)\abs{\eta_\parallel}^2+\abs{\eta_\perp}^2\right).
\end{equation}
Since
$\abs{\eta}^2=\abs{\eta_\parallel}^2+\abs{\eta_\perp}^2$,
placing the smaller of $ 1$ and $p(x)-1 $ below both coefficients and the larger above both coefficients proves \eqref{eq:unreg-bounds}.

$(b)$ If $N\geq2$, a nonzero radial vector and a nonzero tangential vector both exist.  Equality in the two sides of \eqref{eq:unreg-bounds} is attained by choosing whichever of these directions corresponds to the smaller or larger eigenvalue.  Thus the constants are sharp, and the ratio of the larger eigenvalue to the smaller is \eqref{eq:kappa-unreg}. To see this, we distinguish the three cases. \\
\emph{Case 1: $p(x)>2$.} Then $\min\{1,p(x)-1\}=1$ and $ \max\{1,p(x)-1\}= p(x)-1$. Choose a nonzero vector $\eta\in\xi^\perp$. Then
\begin{equation}\label{eq:unreg-quadratic-exact-a}
\eta_\parallel=0, \qquad \eta_\perp=\eta.
\end{equation}

Since \(\eta_\parallel\) belongs to the radial invariant subspace, Lemma \ref{lem:unreg-spectrum} gives
\begin{equation}\label{eq:theorem43-D-parallel}
D(x,\xi)\eta_\parallel = \mu_\parallel(x,\xi)\eta_\parallel. 
\end{equation}
Similarly, because \(\eta_\perp\in\xi^\perp\),
\begin{equation}\label{eq:theorem43-D-perpendicular}
D(x,\xi)\eta_\perp = \mu_\perp(x,\xi)\eta_\perp. 
\end{equation}
Using the linearity of the matrix \(D(x,\xi)\), we obtain
\begin{align}\label{eq:theorem43-D-general-vector}
D(x,\xi)\eta = D(x,\xi)(\eta_\parallel+\eta_\perp) =D(x,\xi)\eta_\parallel + D(x,\xi)\eta_\perp= \mu_\parallel\eta_\parallel + \mu_\perp\eta_\perp. 
\end{align}
For readability, the dependence of the eigenvalues on \((x,\xi)\) is temporarily suppressed in the last line.\\
The quadratic form associated with \(D(x,\xi)\) in the direction \(\eta\) is
\[
\eta^TD(x,\xi)\eta.
\]
Using
\[
\eta=\eta_\parallel+\eta_\perp
\]
and \eqref{eq:theorem43-D-general-vector}, we obtain
\begin{align}\label{eq:theorem43-quadratic-expanded}
\eta^TD(x,\xi)\eta = (\eta_\parallel+\eta_\perp)^T \left( \mu_\parallel\eta_\parallel + \mu_\perp\eta_\perp \right)
= \mu_\parallel \eta_\parallel^T\eta_\parallel + \mu_\perp \eta_\parallel^T\eta_\perp  + \mu_\parallel \eta_\perp^T\eta_\parallel + \mu_\perp \eta_\perp^T\eta_\perp. 
\end{align} Because the parallel and perpendicular components are orthogonal,
\[
\eta_\parallel^T\eta_\perp = \eta_\parallel\cdot\eta_\perp = 0\quad \text{and}\quad\eta_\perp^T\eta_\parallel=0.
\]
Therefore, both cross terms in \eqref{eq:theorem43-quadratic-expanded} vanish. We are left with
\begin{equation}
\eta^TD(x,\xi)\eta = \mu_\parallel|\eta_\parallel|^2 + \mu_\perp|\eta_\perp|^2. \label{eq:theorem43-spectral-quadratic-form}
\end{equation}
Using \eqref{eq:theorem43-spectral-quadratic-form} and \eqref{eq:unreg-quadratic-exact-a}, we obtain
\[
\eta^TD(x,\xi)\eta = |\xi|^{p(x)-2}|\eta|^2.
\]
Thus equality holds in the lower bound. Therefore, the lower constant $1$ cannot be increased.\\
Now choose a nonzero vector $\eta\in\operatorname{span}\{\xi\}$. Then
\[
\eta_\parallel=\eta, \quad \eta_\perp=0.
\]
Therefore,
\[
\eta^TD(x,\xi)\eta = (p(x)-1)|\xi|^{p(x)-2}|\eta|^2.
\]
Thus equality holds in the upper bound. Therefore, the upper constant $p(x)-1$ cannot be decreased.\\

\emph{Case 2: $1<p(x)<2$.} Then
\[
\min\{1,p(x)-1\}=p(x)-1 \quad \text{and} \quad \max\{1,p(x)-1\}=1.
\]
Choose a nonzero vector $\eta\in\operatorname{span}\{\xi\}$. Then
\[
\eta^TD(x,\xi)\eta = (p(x)-1)|\xi|^{p(x)-2}|\eta|^2.
\]
Hence equality holds in the lower bound, so the lower constant $p(x)-1$ cannot be increased. Choose instead a nonzero vector $\eta\in\xi^\perp$. Then
\[
\eta^TD(x,\xi)\eta = |\xi|^{p(x)-2}|\eta|^2.
\] Hence equality holds in the upper bound, so the upper constant $1$ cannot be decreased. \\

\emph{Case 3: $p(x)=2$.} In this case the two displayed eigenvalue factors coincide at $1$. Thus the diffusion matrix is $D(x,\xi)= I$,  which leads to
\[
\eta^TD(x,\xi)\eta = \eta^TI\eta = |\eta|^2.
\]
Both the lower and upper bounds are equalities for every $\eta\in\mathbb R^N$. Thus, for $N\geq2$, the constants in the lemma are sharp in every exponent regime.\\

From Lemma \ref{lem:unreg-spectrum}, the two possible eigenvalues are
\[
\mu_\parallel = (p(x)-1)|\xi|^{p(x)-2} \quad \text{and} \quad \mu_\perp = |\xi|^{p(x)-2}.
\]
Because the common factor $|\xi|^{p(x)-2}$ is strictly positive whenever $\xi \neq 0$, comparing the two eigenvalues is equivalent to comparing \(p(x)-1\) and \(1\). Therefore, the smallest eigenvalue is
\begin{equation}\label{eq:theorem43-min-eigenvalue}
\lambda_{\min}\bigl(D(x,\xi)\bigr) = \min\{1,p(x)-1\}|\xi|^{p(x)-2}, 
\end{equation}
and the largest eigenvalue is
\begin{equation}\label{eq:theorem43-max-eigenvalue}
\lambda_{\max}\bigl(D(x,\xi)\bigr) = \max\{1,p(x)-1\}|\xi|^{p(x)-2}. 
\end{equation}
Since $p(x)>1$ and $\xi\neq0$, both eigenvalues are strictly positive. Thus $D(x,\xi)$ is positive definite at every point where $\xi\neq0$.\\

Since $D(x,\xi)$ is a symmetric positive-definite matrix, the corresponding spectral condition number is
\[
\kappa(D(x,\xi)) = \frac{ \lambda_{\max}(D(x,\xi)) }{ \lambda_{\min}(D(x,\xi)) }.
\]
Using \eqref{eq:theorem43-min-eigenvalue} and \eqref{eq:theorem43-max-eigenvalue}, we obtain
\begin{align}\label{eq:theorem43-condition-first}
\kappa(D(x,\xi)) &= \frac{ \max\{1,p(x)-1\}|\xi|^{p(x)-2} }{ \min\{1,p(x)-1\}|\xi|^{p(x)-2} }. 
\end{align}
Since \(\xi\neq0\), $|\xi|^{p(x)-2}>0$, so this common factor cancels to
\begin{equation}\label{eq:theorem43-condition-second}
\kappa(D(x,\xi)) = \frac{ \max\{1,p(x)-1\} }{ \min\{1,p(x)-1\} }. 
\end{equation}
We can simplify this expression by considering two cases $p(x)\geq2$ and $1 <p(x)<2$ from which \eqref{eq:kappa-unreg} follows.\\
$(c)$ If $N= 1$, $\xi^\perp=\{0\}$, and hence, only the radial eigendirection exists. Thus, the matrix is the scalar in \eqref{eq:N1-D}, whose sole singular value divided by itself equals one.
\end{proof}

\begin{remark}\label{rem:anisotropy-scale}
For $N\geq2$, the radial-to-tangential eigenvalue ratio is exactly
\begin{equation}\label{eq:ratio-pminus1}
\frac{\mu_\parallel}{\mu_\perp}=p(x)-1.
\end{equation}
This ratio records shape, not speed.  The common factor $\abs{\xi}^{p(x)-2}$ controls the absolute diffusion scale and may tend to zero or infinity even when the ratio remains bounded.  This distinction is decisive whenever the profile-frozen diffusion is used to study its stochastic time scale.
\end{remark}
The following classical expansion is used to identify the coefficients of the stochastic construction (see Remark \ref{rem:5.2}); its proof is included for completeness. Note also that because the pointwise unregularized formulae \eqref{eq:unreg-expansion}-\eqref{eq:unreg-expanded-scalar} may be singular or undefined on the critical set $\{x \in\Omega: \nabla u(x)=0\} $, the regularized counterpart \eqref{eq:reg-expansion}-\eqref{eq:reg-expanded-scalar} used in the stochastic construction is stated and proved separately.
\begin{lemma}\label{thm:unreg-expansion}
Under Assumption \ref{ass:classical-deterministic}, at every point $x$ with $\nabla u(x)\neq0$,
\begin{equation}\label{eq:unreg-expansion}
 \mathrm{div} A(x,\nabla u)=D(x,\nabla u)\fro D^2u+c(x,\nabla u)\cdot\nabla u,
\end{equation}
where
\begin{equation}\label{eq:c-unreg}
c(x,\xi)=\abs{\xi}^{p(x)-2}\log\abs{\xi}\,\nabla p(x) \quad\xi\neq 0.
\end{equation}
Equivalently,
\begin{align}\label{eq:unreg-expanded-scalar}
\mathrm{div} A(x,\nabla u)=\abs{\nabla u}^{p(x)-2}\Bigl[\Delta u+(p(x)-2)\frac{\ip{D^2u\nabla u}{\nabla u}}{\abs{\nabla u}^2}+\log\abs{\nabla u}\,\nabla p\cdot\nabla u\Bigr].
\end{align}
\end{lemma}
\begin{proof}
Writing $A_i(x,\nabla u(x))=|\nabla u|^{p(x)-2}u_{x_i}$ and then applying the chain rule gives
\[
\frac{\partial}{\partial x_i}A_i(x,\nabla u(x))=\left.\frac{\partial A_i}{\partial x_i}(x,\xi)\right|_{\xi=\nabla u}+\sum_{j=1}^N\frac{\partial A_i}{\partial\xi_j}(x,\nabla u)u_{x_jx_i}.
\]
Summing over $ i$, the second term is
\[
\sum_{i,j}D_{ij}(x,\nabla u)u_{x_ix_j}=D(x,\nabla u):D^2u.
\]
Since
\[
\frac{\partial}{\partial x_i}|\xi|^{p(x)-2}=|\xi|^{p(x)-2}\log|\xi|\,p_{x_i}(x),
\]
we obtain
\[
\left.\frac{\partial A_i}{\partial x_i}(x,\xi)\right|_{\xi=\nabla u}=|\nabla u|^{p(x)-2}\log|\nabla u|p_{x_i}u_{x_i}.
\]
Summing over $i$ yields $c(x,\nabla u)\cdot\nabla u$. This proves \eqref{eq:unreg-expansion}.  Substituting \eqref{eq:D-unreg} and contracting gives
\[
 D(x,\nabla u)\fro D^2u=\abs{\nabla u}^{p(x)-2}\Delta u+(p(x)-2)\abs{\nabla u}^{p(x)-4}(D^2u\,\nabla u)\cdot\nabla u ,
\]
which along with the term $c(x,\nabla u)\cdot\nabla u$ provides \eqref{eq:unreg-expanded-scalar}.
\end{proof}
\begin{remark}\label{rem:5.2}
The unregularized fields $D(x,\nabla u)$ and $c(x,\nabla u)$ identify the second-order and first-order terms away from critical points. The stochastic construction uses their regularized
counterparts $D_\varepsilon(x,\nabla u)$ and $c_\varepsilon(x,\nabla u)$ as the second-order coefficient and drift, respectively.  
\end{remark}

\section{Regularized tangent geometry and frozen coefficients}\label{sec:regularization}

Recall
\begin{equation*}\label{eq:r-eps}
r_\varepsilon(\xi)=\sqrt{\varepsilon^2+\abs{\xi}^2},\quad A_\varepsilon(x,\xi)=r_\varepsilon(\xi)^{p(x)-2}\xi.
\end{equation*}
We write  $r_\varepsilon$ when its argument is clear.  Since $r_\varepsilon(\xi)\geq\varepsilon>0$, all coefficient formulas below are defined at $\xi =0$. We now present the derivative and spectrum of the regularized flux.

\begin{lemma}\label{lem:reg-spectrum}
Let $\varepsilon>0$, fix $x\in\Omega$ with $p(x)>1$, and let $\xi\in\mathbb{R}^N$.
Then
\begin{equation}\label{eq:D-eps}
 D_\varepsilon(x,\xi):=D_\xi A_\varepsilon(x,\xi)=r_\varepsilon(\xi)^{p(x)-2}\Id+(p(x)-2)r_\varepsilon(\xi)^{p(x)-4}\xi\otimes\xi.
\end{equation}
If $N\geq2$ and $\xi\neq 0$, the tangential eigenvalue on $\xi^\perp$ is
\begin{equation}\label{eq:mu-perp-eps}
\mu_{\perp,\varepsilon}=r_\varepsilon(\xi)^{p(x)-2}.
\end{equation}
For every $N\geq1$ and $\xi\neq 0$, the radial eigenvalue on $\Span\{\xi\}$ is
\begin{align}
\mu_{\parallel,\varepsilon} &=r_\varepsilon(\xi)^{p(x)-2}\gamma_\varepsilon(x,\xi),\label{eq:mu-parallel-eps}\\
\gamma_\varepsilon(x,\xi)&:=1+(p(x)-2)\theta_\varepsilon(\xi),
 \quad
\theta_\varepsilon(\xi):=\frac{\abs{\xi}^2}{\varepsilon^2+\abs{\xi}^2}.
 \label{eq:gamma-theta}
\end{align}
At $\xi=0$,
\begin{equation}\label{eq:D-eps-zero}
 D_\varepsilon(x,0)=\varepsilon^{p(x)-2}\Id.
\end{equation}
For $\xi\neq0$, the radial and tangential subspaces are invariant. When $N\geq2$ and $p(x)\neq 2$, they are the full eigenspaces, with multiplicities $1$ and $N-1$.  If $p(x)= 2$, then $D_\varepsilon=\Id$ and its sole eigenvalue is $1$, of multiplicity $N$.  If $N=1$, only the radial eigenvalue is present. At $\xi=0$, the scalar $\varepsilon^{p(x)-2}$ is the sole eigenvalue and has multiplicity $N$.
\end{lemma}
\begin{proof}
Because $\partial_{\xi_j}r_\varepsilon=\xi_j/r_\varepsilon$, the product and chain rules applied to $A_{\varepsilon,i}=r_\varepsilon^{p(x)-2}\xi_i$ give
\[
\partial_{\xi_j}A_{\varepsilon,i}=r_\varepsilon^{p(x)-2}\delta_{ij}+(p(x)-2)r_\varepsilon^{p(x)-4}\xi_i\xi_j,
\]
which is \eqref{eq:D-eps}.  If $N\geq2$, $\xi\neq0$, and $0\neq\eta\in\xi^\perp$, the tensor term vanishes and gives
\eqref{eq:mu-perp-eps}.
If $\eta\parallel\xi$, then $(\xi\otimes\xi)\eta=|\xi|^2 \eta$, giving
\[
D_\varepsilon \eta=\left(r_\varepsilon^{p(x)-2}+(p(x)-2)r_\varepsilon^{p(x)-4}|\xi|^2\right)\eta.
\]
Factoring $r_\varepsilon^{p(x)-2}$ proves \eqref{eq:mu-parallel-eps}. For $\xi\neq 0$, $\theta_\varepsilon(\xi)>0$, so the radial and tangential factors are distinct exactly when $p(x)\neq2$.  Orthogonal decomposition then gives the stated multiplicities when $N\geq2$; when $N=1$, the radial subspace is all of $\mathbb{R}$.  If $p(x)=2$, \eqref{eq:D-eps} directly gives $D_\varepsilon=\Id$.
Finally, setting $\xi=0$ in \eqref{eq:D-eps} gives \eqref{eq:D-eps-zero}, including its multiplicity assertion.
\end{proof}

\begin{remark}
At $\xi=0$, formula \eqref{eq:D-eps-zero} shows that the matrix is a positive scalar multiple of $\Id$, so every direction has the same eigenvalue. Thus regularization not only prevents vanishing or blow-up for fixed $\varepsilon$, it also removes directional anisotropy exactly at $\xi=0$; that is, there is no distinguished radial direction at $\xi=0$.
\end{remark}

We next establish explicit ellipticity estimates on bounded gradient ranges.\\

For $M>0$, set
\begin{equation}\label{eq:R-eps-M}
 R_{\varepsilon,M}=\sqrt{\varepsilon^2+M^2}.
\end{equation}
Define the bounds
\begin{align}
m_{\varepsilon,M}&=\min\left\{\varepsilon^{p^--2},\varepsilon^{p^+-2},R_{\varepsilon,M}^{p^--2},R_{\varepsilon,M}^{p^+-2}\right\},\label{eq:m-eps-M}\\
 M_{\varepsilon,M}&=\max\left\{\varepsilon^{p^--2},\varepsilon^{p^+-2},R_{\varepsilon,M}^{p^--2},R_{\varepsilon,M}^{p^+-2}\right\},\label{eq:M-eps-M}
\end{align}
and
\begin{equation}\label{eq:p-factors}
\underline q=\min\{1,p^--1\},
 \quad
\overline q=\max\{1,p^+-1\}.
\end{equation}

\begin{lemma}\label{thm:reg-uniform}
Assume Assumption \ref{ass:exponent-range}, fix $\varepsilon>0$ and $M>0$, and let $\abs{\xi}\leq M$. Then, for every $x\in\Omega$ and $\eta\in\mathbb{R}^N$,
\begin{equation}\label{eq:uniform-explicit}
\alpha_{\varepsilon,M}\abs{\eta}^2\leq \eta^\top D_\varepsilon(x,\xi)\eta\leq \beta_{\varepsilon,M}\abs{\eta}^2,
\end{equation}
where
\begin{equation}\label{eq:alpha-beta-explicit}
\alpha_{\varepsilon,M}=\underline q\,m_{\varepsilon,M}>0,
 \quad
 \beta_{\varepsilon,M}=\overline q\,M_{\varepsilon,M}<\infty.
\end{equation}
\end{lemma}

\begin{proof}
For $\xi\neq0$, the radial eigenvalue has the form below; the displayed scalar factors also extend to $\xi=0$. Since $0\leq\theta_\varepsilon(\xi)=|\xi|^2/(\varepsilon^2+|\xi|^2)<1$, we have
\begin{equation}\label{eq:lemma62-radial-factorization}
\mu_{\parallel,\varepsilon}(x,\xi) = r_\varepsilon(\xi)^{p(x)-2} \gamma_{\varepsilon}(x,\xi),
\end{equation}
where
\begin{equation}\label{eq:lemma62-B-definition}
\gamma_{\varepsilon}(x,\xi) := 1+(p(x)-2)\theta_\varepsilon(\xi).
\end{equation}
The identity
\[
 \gamma_\varepsilon(x,\xi)=(1-\theta_\varepsilon(\xi))\,1+\theta_\varepsilon(\xi)\,(p(x)-1)
\]
expresses $\gamma_\varepsilon$ as a convex combination of $1$ and $p(x)-1$. Consequently,
\[
\min\{1,p(x)-1\}\leq\gamma_\varepsilon(x,\xi)\leq\max\{1,p(x)-1\}.
\]
Using $p^-\leq p(x)\leq p^+$ gives
\begin{equation}\label{eq:lemma62-B-definition-a}
\underline q \leq \gamma_{\varepsilon}(x,\xi) \leq \overline q,
\end{equation}
for every  $x\in\Omega$ and $|\xi|\leq M$. Define the scalar function
\begin{equation} \label{eq:lemma62-F-definition}
F(q,r):=r^{q-2},
\end{equation}
with $r:=r_\varepsilon(\xi)$ and $q:=p(x)\in[p^-,p^+]$. Since $\varepsilon\leq r_{\varepsilon}(\xi)\leq R_{\varepsilon,M}$, the pair $(q, r)$ belongs to the rectangle
\begin{equation}\label{eq:lemma62-compact-rectangle}
K_{\varepsilon,M} = [p^-,p^+] \times [\varepsilon,R_{\varepsilon,M}].
\end{equation}
Since $F$ is continuous on the compact set $K_{\varepsilon,M}$, it attains a minimum and a maximum on this rectangle
\begin{equation}\label{eq:lemma62-m-definition}
m_{\varepsilon,M} := \min_{ \substack{ (q, r)\in K_{\varepsilon,M} } } r^{q-2} 
\end{equation}
and
\begin{equation} \label{eq:lemma62-M-definition}
M_{\varepsilon,M} := \max_{ \substack{ (q, r)\in K_{\varepsilon,M}} } r^{q-2}.
\end{equation}

For each fixed $r>0$, the map $q\mapsto r^{q-2}$ is monotone (or constant if $r=1$), so its extrema over $[p^-,p^+]$ occur at endpoints. For either endpoint $q$, the map $r\mapsto r^{q-2}$ is again monotone (or constant if $q=2$), so its extrema over $[\varepsilon,R_{\varepsilon,M}]$ occur at endpoints.  Thus $F$ assumes its extrema on the rectangle $K_{\varepsilon,M}$  among its four corners, even when the ranges cross $q=2$ or $r=1$.  In particular,

\begin{equation}\label{eq:lemma62-explicit-m}
m_{\varepsilon,M} = \min \left\{ \varepsilon^{p^--2}, \varepsilon^{p^+-2}, R_{\varepsilon,M}^{p^--2}, R_{\varepsilon,M}^{p^+-2} \right\}
\end{equation}
and
\begin{equation} \label{eq:lemma62-explicit-M}
M_{\varepsilon,M} = \max \left\{ \varepsilon^{p^--2}, \varepsilon^{p^+-2}, R_{\varepsilon,M}^{p^--2}, R_{\varepsilon,M}^{p^+-2} \right\}.
\end{equation}
Therefore
\begin{equation}\label{eq:lemma62-explicit-common-bounds}
0< m_{\varepsilon,M} \leq r_\varepsilon(\xi)^{p(x)-2} \leq M_{\varepsilon,M} <\infty.
\end{equation}

For this estimate write $\mu_{\perp,\varepsilon}=r_\varepsilon^{p(x)-2}$ and $\mu_{\parallel,\varepsilon}=r_\varepsilon^{p(x)-2}\gamma_\varepsilon$ for the two scalar factors, also at $\xi=0$.  The first factor is an actual tangential eigenvalue when $N\geq2$ and $\xi\neq0$; for $N=1$ it is only an auxiliary scalar. Using $0<\underline q\leq1\leq \overline q$,
\eqref{eq:lemma62-B-definition-a}, and \eqref{eq:lemma62-explicit-common-bounds}, we obtain

\begin{equation}\label{eq:lemma62-tangential-common-bounds}
\alpha_{\varepsilon,M}\leq \mu_{\perp,\varepsilon}(x,\xi) \leq  \beta_{\varepsilon,M},
\end{equation}
and
\begin{equation}\label{eq:lemma62-tangential-common-bounds-a}
\alpha_{\varepsilon,M}\leq \mu_{\parallel,\varepsilon}(x,\xi) \leq  \beta_{\varepsilon,M},
\end{equation}
respectively. The eigenvalue bounds from \eqref{eq:lemma62-tangential-common-bounds} and \eqref{eq:lemma62-tangential-common-bounds-a} can now be summarized as
\begin{equation}\label{eq:lemma62-all-eigenvalue-bounds}
\alpha_{\varepsilon,M} \leq \mu \leq \beta_{\varepsilon,M} 
\end{equation}
for every eigenvalue $\mu$ of $D_\varepsilon(x,\xi)$, uniformly for $x\in\Omega,\,\,  |\xi|\leq M$.\\
Next suppose $\xi\neq0$ and convert these scalar bounds into quadratic-form bounds. We start to construct the regularized counterpart of \eqref{eq:theorem43-spectral-quadratic-form}; that is, we shall obtain
\begin{align} \label{eq:lemma62-quadratic-spectral}
\eta^TD_\varepsilon(x,\xi)\eta &= \mu_{\perp,\varepsilon} |\eta_\perp|^2 +\mu_{\parallel,\varepsilon} |\eta_\parallel|^2.
\end{align}
Starting from
\[
D_\varepsilon(x,\xi) =r_\varepsilon^{p(x)-2}I + (p(x)-2) r_\varepsilon^{p(x)-4} \xi\otimes\xi,
\]
we obtain
\begin{align*}
\eta^TD_\varepsilon(x,\xi)\eta = r_\varepsilon^{p(x)-2}|\eta|^2+ (p(x)-2) r_\varepsilon^{p(x)-4} (\xi\cdot\eta)^2.
\end{align*}
For $\xi\neq0$, let
\[
\eta_\parallel = \frac{\xi\cdot\eta}{|\xi|^2}\xi.
\]
Then
\[
|\eta_\parallel|^2 = \frac{(\xi\cdot\eta)^2}{|\xi|^2},
\]
so
\[
(\xi\cdot\eta)^2 = |\xi|^2|\eta_\parallel|^2.
\] Therefore,
\begin{align*}
\eta^TD_\varepsilon(x,\xi)\eta = r_\varepsilon^{p(x)-2}|\eta|^2 + (p(x)-2) r_\varepsilon^{p(x)-4} |\xi|^2|\eta_\parallel|^2.
\end{align*}
Using the identity
\[
|\eta|^2 = |\eta_\parallel|^2+|\eta_\perp|^2,
\]
we find
\begin{align*}
\eta^TD_\varepsilon(x,\xi)\eta = r_\varepsilon^{p(x)-2} |\eta_\perp|^2 + r_\varepsilon^{p(x)-2} \left( 1+ (p(x)-2) \frac{|\xi|^2} {\varepsilon^2+|\xi|^2} \right) |\eta_\parallel|^2,
\end{align*}
which is \eqref{eq:lemma62-quadratic-spectral}.\\
Since both  scalar factors are at least \(\alpha_{\varepsilon,M}\),
\begin{align}\label{eq:lemma62-quadratic-lower}
\eta^TD_\varepsilon(x,\xi)\eta &\geq \alpha_{\varepsilon,M} |\eta_\parallel|^2 + \alpha_{\varepsilon,M} |\eta_\perp|^2  =\alpha_{\varepsilon,M}|\eta|^2.
\end{align}
Similarly, both  scalar factors  are at most $\beta_{\varepsilon,M}$,
\begin{align} \label{eq:lemma62-quadratic-upper}
\eta^TD_\varepsilon(x,\xi)\eta &\leq \beta_{\varepsilon,M} |\eta_\parallel|^2 + \beta_{\varepsilon,M} |\eta_\perp|^2 = \beta_{\varepsilon,M}|\eta|^2.
\end{align}
These calculations also apply when $N=1$, since then $\eta_\perp=0$; they do not require a tangential eigenvalue. At $\xi=0$,
\[
D_\varepsilon(x,0) = \varepsilon^{p(x)-2}I,
\]
and the same inequalities follow directly from the common-factor bounds.\\
Therefore,
\[
\alpha_{\varepsilon,M}|\eta|^2 \leq \eta^TD_\varepsilon(x,\xi)\eta \leq \beta_{\varepsilon,M}|\eta|^2
\]
for every $x\in\Omega, \,\, |\xi|\leq M, \,\,\eta\in\mathbb R^N$. This proves uniform ellipticity of $D_\varepsilon(x,\xi)$ on bounded gradient ranges.
\end{proof}
\begin{remark}\label{rem:remark 64}
Since \eqref{eq:lemma62-all-eigenvalue-bounds} holds uniformly for $x\in\Omega,\, |\xi|\leq M$, all eigenvalues of $D_\varepsilon(x,\xi)$ are strictly positive. Thus $D_\varepsilon(x,\xi)$ belongs to the open cone of real symmetric positive-definite $N\times N$ matrices $\operatorname{Sym}_N^{++}$; that is, $D_\varepsilon(x,\xi) \in\operatorname{Sym}_N^{++}$.
\end{remark}

\begin{lemma}\label{prop:reg-limit}
Assume $1<p^-\leq p(x)\leq p^+<\infty$ and fix $x\in\Omega$.
\begin{enumerate}[label=(\alph*)]
\item If $N\geq2$ and $\xi\neq 0$, then
\begin{equation}\label{eq:kappa-eps}
\kappa(D_\varepsilon(x,\xi))=\max\{\gamma_\varepsilon(x,\xi),\gamma_\varepsilon(x,\xi)^{-1}\}.
\end{equation}
At $\xi=0$, and for every $\xi$ when $N=1$, the condition number is one.
\item Uniformly in $x$ and $\xi$,
\begin{equation}\label{eq:kappa-eps-global}
 \kappa(D_\varepsilon(x,\xi))\leq \max\left\{p^+-1,\frac1{p^--1}\right\}\quad(N\geq2).
\end{equation}
\item For each fixed $\xi\neq 0$,
\begin{equation}\label{eq:D-eps-limit}
 D_\varepsilon(x,\xi)\rightarrow D(x,\xi)\quad(\varepsilon\downarrow0).
\end{equation}
For every $0<\delta\leq M<\infty$, the convergence is uniform for all $x\in\Omega$ and $\delta\leq\abs{\xi}\leq M$. The stated exponent bounds alone do not ensure a finite extension
of this convergence across $\xi=0$.
\end{enumerate}
\end{lemma}

\begin{proof}
$(a)$ For $N\geq2$ and $\xi\neq0$, the spectrum consists of $\mu_{\perp,\varepsilon}$ and $\mu_{\parallel,\varepsilon}=\gamma_\varepsilon\mu_{\perp,\varepsilon}$.  Their ratio gives \eqref{eq:kappa-eps}. At $\xi=0$, the matrix is a scalar multiple of the identity. In $N=1$, there is only one eigenvalue, so the condition number is one.\\
$(b)$ Formula \eqref{eq:lemma62-B-definition-a} gives
\[
\max\{\gamma_\varepsilon,\gamma_\varepsilon^{-1}\}\leq \max\{\overline q,\underline q^{-1}\}=\max\left\{1,p^+-1,\frac1{p^--1}\right\}.
\]
At least one of $p^+-1$ and $(p^--1)^{-1}$ is at least $1$: if $p^+\geq2$, the first has this property, while if $p^+<2$, then $p^-\leq p^+<2$ and the second has it. The extra $1$ may
therefore be omitted, proving \eqref{eq:kappa-eps-global}. At $\xi=0$ the condition number equals $1$ and obeys the same bound.\\
$(c)$
For fixed $\xi\neq 0$, $r_\varepsilon(\xi)\to\abs{\xi}$ and $r_\varepsilon(\xi)^{p(x)-4}\to\abs{\xi}^{p(x)-4}$, so \eqref{eq:D-eps} converges pointwise term by term to \eqref{eq:D-unreg}. 
For uniform convergence, define
\[
H(t,q,\xi)=(t+\abs{\xi}^2)^{(q-2)/2}\Id+(q-2)(t+\abs{\xi}^2)^{(q-4)/2}\xi\otimes\xi.
\]
This matrix-valued map is uniformly continuous on the compact set
\[
[0,1]\times[p^-,p^+]\times\{\xi\in\mathbb{R}^N:\delta\leq\abs{\xi}\leq M\},
\]
because $\delta>0$.  Moreover, $D_\varepsilon(x,\xi)=H(\varepsilon^2,p(x),\xi)$ and $D(x,\xi)=H(0,p(x),\xi)$. Hence
\[
\sup_{\substack{x\in\Omega\\\delta\leq\abs{\xi}\leq M}}\norm{D_\varepsilon(x,\xi)-D(x,\xi)}_{\mathrm{op}}\rightarrow0.
\]
No continuity of $p$ is needed for this conclusion. At $\xi=0$, however, $D_\varepsilon(x,0)=\varepsilon^{p(x)-2}\Id$, which tends to zero for $p(x)>2$, equals $\Id$ for $p(x)=2$, and diverges for $1<p(x)<2$.  A single uniform limit across the critical set $\{\xi \in\mathbb R^N: \xi=0\}$ is therefore impossible in general.
\end{proof}
\begin{lemma}\label{thm:reg-expansion}
Under Assumption \ref{ass:classical-deterministic}, at every $x\in\Omega$,
\begin{equation}\label{eq:reg-expansion}
\mathrm{div} A_\varepsilon(x,\nabla u)=c_\varepsilon(x,\nabla u)\cdot\nabla u+D_\varepsilon(x,\nabla u)\fro D^2u,
\end{equation}
where
\begin{equation}\label{eq:c-eps}
c_\varepsilon(x,\xi)=r_\varepsilon(\xi)^{p(x)-2}\log r_\varepsilon(\xi)\,\nabla p(x).
\end{equation}
Equivalently,
\begin{align}\label{eq:reg-expanded-scalar}
\mathrm{div} A_\varepsilon(x,\nabla u)=r_\varepsilon(\nabla u)^{p(x)-2}\Bigg[&\Delta u+(p(x)-2)\frac{\ip{D^2u\nabla u}{\nabla u}}
{\varepsilon^2+\abs{\nabla u}^2}+\log r_\varepsilon(\nabla u)\,\nabla p\cdot\nabla u\Bigg].
\end{align}
\end{lemma}

\begin{proof}
The proof follows the chain-rule calculation in Lemma \ref{thm:unreg-expansion}, but no exclusion of $\nabla u=0$ is needed. For fixed $\xi$,
\begin{align*}
\partial_{x_i}r_\varepsilon(\xi)^{p(x)-2}&=r_\varepsilon(\xi)^{p(x)-2}\log r_\varepsilon(\xi)\,p_{x_i}(x).
\end{align*}
Thus the explicit $x$-derivatives sum to $c_\varepsilon(x,\nabla u)\cdot\nabla u$. On the other hand, differentiation through $\nabla u$ contributes $D_\varepsilon(x,\nabla u)\fro D^2u$, proving \eqref{eq:reg-expansion}. Lastly, contracting the two terms of \eqref{eq:D-eps} with $D^2u$ gives \eqref{eq:reg-expanded-scalar}.
\end{proof}
\begin{corollary}\label{cor:critical-reg-expansion}
Under Assumption~\ref{ass:classical-deterministic}, if $\nabla u(x)=0$, then
\begin{equation}\label{eq:critical-reg-value}
\mathrm{div} A_\varepsilon(x,\nabla u)=\varepsilon^{p(x)-2}\Delta u.
\end{equation}
The first-order contribution $c_\varepsilon(x,\nabla u)\cdot\nabla u$ vanishes at the critical point $\nabla u(x) =0 $, while the drift vector itself need not vanish
\[
c_\varepsilon(x,0)=\varepsilon^{p(x)-2}\log\varepsilon\,\nabla p(x).
\]
\end{corollary}
\begin{proof}
Set $\nabla u=0$ in \eqref{eq:reg-expansion}. By \eqref{eq:D-eps-zero}, the second-order term is $\varepsilon^{p(x)-2}\Delta u$, and the first-order term is zero.
\end{proof}

\begin{lemma}\label{lem:lipschitz-coeff}
Under Assumption \ref{ass:classical-framework}, there are constants $0<\alpha_u\leq \beta_u<\infty$ such that
\begin{equation}\label{eq:frozen-ellipticity}
\alpha_u\abs{\eta}^2\leq\eta^\top D_{u,\varepsilon}(x)\eta\leq\beta_u\abs{\eta}^2\quad x\in\overline\Omega,\ \eta\in\mathbb{R}^N.
\end{equation}
The fields $D_{u,\varepsilon}$, $b_{u,\varepsilon}$, and $\sigma_{u,\varepsilon}$ are bounded on $\overline\Omega$ and locally Lipschitz on a $\mathcal{U}$. For $\nabla u(x)\neq0$,
\begin{equation}\label{eq:sigma-spectral}
\sigma_{u,\varepsilon}(x)=\sqrt{2\mu_{\parallel,\varepsilon}}P_{\nabla u}+\sqrt{2\mu_{\perp,\varepsilon}}(\Id-P_{\nabla u}),
\end{equation}
and at $\nabla u(x)=0$,
\begin{equation}\label{eq:sigma-zero}
\sigma_{u,\varepsilon}(x)=\sqrt2\,\varepsilon^{(p(x)-2)/2}\Id.
\end{equation}
\end{lemma}

\begin{proof}
We prove separately that
\[
b_{u,\varepsilon}(x) = c_\varepsilon(x,\nabla u(x))
\]
and the diffusion coefficient
\[
\sigma_{u,\varepsilon}(x) = \left(2D_\varepsilon(x,\nabla u(x)) \right )^{1/2}
\]
are locally Lipschitz in a neighborhood of $\overline\Omega$, and that both maps are bounded on $\overline\Omega$. We also establish the asserted bounds for $D_{u,\varepsilon}$. All matrix estimates in this proof use the operator norm; the Frobenius-norm consequences needed for stochastic integrals follow from \eqref{eq:matrix-norm-equivalence}.\\
Recall the fixed bounded open set $V$ in \eqref{eq:framework-neighborhood}, so that
\[
\overline\Omega\subset V,\quad \overline V\Subset\mathcal{U}.
\]
Since $u\in C^{2,1}(\mathcal U)$, the map $\nabla u$ is $C^1$, hence locally Lipschitz, on $\mathcal U$. Thus $\nabla u|_{\overline V}$ is Lipschitz, and hence, there is a constant $L_{\nabla u}>0$ such that
\begin{equation}\label{eq:lemma72-gradient-Lipschitz}
|\nabla u(x)-\nabla u(y)| \leq L_{\nabla u}|x-y| \quad \text{for all }x,y\in\overline V. 
\end{equation}

The bound $M_u$ fixed in \eqref{eq:framework-neighborhood} satisfies
\begin{equation}\label{eq:lemma72-gradient-bound}
M_u = \sup_{x\in\overline V}|\nabla u(x)|<\infty.
\end{equation}

Recall the compact product set from \eqref{eq:framework-product-set}
\begin{equation}\label{eq:lemma72-compact-set}
K=\overline V\times\{\xi\in\mathbb R^N:|\xi|\leq M_u\}\Subset\mathcal{U}\times\mathbb R^N.
\end{equation}
This definition includes $M_u=0$, when its second factor is $\{0\}$. Every $(x,\nabla u(x))$ with $x\in\overline V$ belongs to $K$. All two-variable Lipschitz estimates below are asserted on this exact set.\\
By Assumption \ref{ass:framework}, both $p$ and $\nabla p$ are locally Lipschitz on $\mathcal{U}$. Thus $p|_{\overline V}$ and $\nabla p|_{\overline V}$ are Lipschitz, giving the constants $L_p,\,L_{\nabla p}>0$ such that
\begin{equation}\label{eq:lemma72-gradient-p-Lipschitz}
|\nabla p(x)-\nabla p(y)| \leq L_{\nabla p}|x-y| \quad \text{for all }x,y\in\overline V. 
\end{equation}
Since $p$ and $\nabla p$ are continuous on the compact set $\overline V$, both are bounded.\\
Define
\[
P_0 := \sup_{x\in\overline V}|p(x)|, \quad P_1 := \sup_{x\in\overline V}|\nabla p(x)|.
\]
Moreover, because $ p\in C^1(\overline V)$, the Mean Value Theorem (MVT) gives
\begin{equation}\label{eq:lemma72-p-Lipschitz}
|p(x)-p(y)| \leq P_1|x-y| \quad \text{for all }x,y\in\overline V. 
\end{equation}
Thus both $p$ and $\nabla p$ are Lipschitz on $\overline V$.\\
For $(x,\xi)\in K$, we have $ |\xi|\leq M_u$. Therefore,
\[
\varepsilon \leq r_\varepsilon(\xi) = \sqrt{\varepsilon^2+|\xi|^2} \leq \sqrt{\varepsilon^2+M_u^2}.
\]
Define
\begin{equation}\label{eq:lemma72-Ru}
R_u := \sqrt{\varepsilon^2+M_u^2}. 
\end{equation}
Then
\begin{equation}\label{eq:lemma72-r-bounds}
\varepsilon \leq r_\varepsilon(\xi) \leq R_u \quad \text{for every }(x,\xi)\in K. 
\end{equation}
In particular, the regularization prevents $r_\varepsilon(\xi)$ from approaching zero. The derivative of $r_\varepsilon$ with respect to $\xi_j$ is
$\frac{\partial r_\varepsilon}{\partial\xi_j} = \frac{\xi_j}{r_\varepsilon(\xi)}$. Therefore,
\begin{equation}\label{r-deriv}
|\nabla_\xi r_\varepsilon(\xi)| = \frac{|\xi|}{r_\varepsilon(\xi)} \leq1.
\end{equation}
Therefore by the MVT,
\begin{equation}\label{eq:lemma72-r-Lipschitz}
|r_\varepsilon(\xi)-r_\varepsilon(\eta)| \leq |\xi-\eta| \quad \text{for all }\xi,\eta\in\mathbb R^N. 
\end{equation}
Thus $r_\varepsilon$ is globally Lipschitz with Lipschitz constant at most $1$.\\
Define
\begin{equation}\label{eq:lemma72-F-definition}
F(x,\xi) := r_\varepsilon(\xi)^{p(x)-2}. 
\end{equation}
Because \(r_\varepsilon(\xi)>0\), we may write
\[
F(x,\xi) = \exp\left( (p(x)-2)\log r_\varepsilon(\xi) \right).
\]
For a spatial coordinate $x_k$,
\begin{align}\label{eq:lemma72-F-x-derivative}
\frac{\partial F}{\partial x_k}(x,\xi) = r_\varepsilon(\xi)^{p(x)-2} \log r_\varepsilon(\xi)\, p_{x_k}(x). 
\end{align}
For a gradient coordinate $\xi_j$,
\begin{align}\label{eq:lemma72-F-xi-derivative}
\frac{\partial F}{\partial\xi_j}(x,\xi) = (p(x)-2) r_\varepsilon(\xi)^{p(x)-4}\xi_j. 
\end{align}
All quantities appearing in \eqref{eq:lemma72-F-x-derivative} and \eqref{eq:lemma72-F-xi-derivative} are continuous on $K$. Because $K$ is compact, there is a constant $C_F>0$ such that \begin{equation}\label{eq:lemma72-F-derivative-bound}
|\nabla_xF(x,\xi)| + |\nabla_\xi F(x,\xi)| \leq C_F \quad \text{for every }(x,\xi)\in K. 
\end{equation}
The derivative formulas show that $F$ is $C^1$ on the open set $\mathcal U\times\mathbb R^N$, hence locally Lipschitz there.
Therefore, $F$ is Lipschitz on $K$ yielding a constant $L_F>0$ such that
\begin{equation}\label{eq:lemma72-F-Lipschitz}
|F(x,\xi)-F(y,\eta)| \leq L_F \left( |x-y|+|\xi-\eta| \right)  \quad \text{for every }(x,\xi),(y,\eta)\in K.\\
\end{equation}

Define
\begin{equation}\label{eq:lemma72-G-definition}
G(x,\xi) := r_\varepsilon(\xi)^{p(x)-2} \log r_\varepsilon(\xi). 
\end{equation}
Equivalently,
\[
G(x,\xi)=F(x,\xi)\log r_\varepsilon(\xi).
\]
Therefore,
\begin{align}\label{eq:lemma72-G-x-derivative}
\frac{\partial G}{\partial x_k}(x,\xi) = r_\varepsilon(\xi)^{p(x)-2} \left( \log r_\varepsilon(\xi) \right)^2 p_{x_k}(x), 
\end{align}
and
\begin{align}\label{eq:lemma72-G-xi-derivative}
\frac{\partial G}{\partial\xi_j}(x,\xi) = r_\varepsilon(\xi)^{p(x)-4}\xi_j \left[ 1+(p(x)-2)\log r_\varepsilon(\xi) \right]. 
\end{align}
The derivatives in \eqref{eq:lemma72-G-x-derivative} and \eqref{eq:lemma72-G-xi-derivative} are continuous on $\mathcal U\times\mathbb R^N$ and bounded on $K$. Thus $G$ is locally Lipschitz on that open set. Therefore, $G$ has bounded first derivatives on $K$ giving a constant $L_G>0$ such that
\begin{equation} \label{eq:lemma72-G-Lipschitz}
|G(x,\xi)-G(y,\eta)| \leq L_G \left( |x-y|+|\xi-\eta| \right)\quad \text{for every }(x,\xi),(y,\eta)\in K
\end{equation}
In particular, $G$ is bounded on $K$. Define
\begin{equation}\label{eq:lemma72-G-bound}
C_G := \sup_{(x,\xi)\in K}|G(x,\xi)|. 
\end{equation}
The regularized drift is
\[
b_{u,\varepsilon}(x) = G(x,\nabla u(x))\nabla p(x).
\]
Let $x,y\in\overline V$. Adding and subtracting $G(y,\nabla u(y))\nabla p(x)$ gives
\begin{align}\label{eq:lemma72-drift-difference}
b_{u,\varepsilon}(x)-b_{u,\varepsilon}(y) &= G(x,\nabla u(x))\nabla p(x)- G(y,\nabla u(y))\nabla p(y) \nonumber\\
&= \left[ G(x,\nabla u(x)) - G(y,\nabla u(y)) \right]\nabla p(x) + G(y,\nabla u(y)) \left[ \nabla p(x)-\nabla p(y) \right]. 
\end{align}
Taking norms and using the triangle inequality gives
\begin{align}\label{eq:lemma72-drift-bound-first}
|b_{u,\varepsilon}(x)-b_{u,\varepsilon}(y)| \leq \left| G(x,\nabla u(x)) - G(y,\nabla u(y)) \right| |\nabla p(x)| + |G(y,\nabla u(y))| |\nabla p(x)-\nabla p(y)|. 
\end{align}
Using
\[
|\nabla p(x)|\leq P_1 \quad \text{and} \quad |G(y,\nabla u(y))|\leq C_G,
\]
and the Lipschitz estimate for $G$, we obtain
\begin{align}\label{eq:lemma72-G-profile-difference}
\left| G(x,\nabla u(x)) - G(y,\nabla u(y)) \right| \leq L_G \left( |x-y| + |\nabla u(x)-\nabla u(y)| \right ). 
\end{align}
Using the Lipschitz property of $\nabla u$,
\[
|\nabla u(x)-\nabla u(y)| \leq L_{\nabla u}|x-y|,
\]
we find
\begin{align}\label{eq:lemma72-G-profile-Lipschitz}
\left| G(x,\nabla u(x)) - G(y,\nabla u(y)) \right| \leq L_G(1+L_{\nabla u})|x-y|. 
\end{align}
Similarly,
\[
|\nabla p(x)-\nabla p(y)| \leq L_{\nabla p}|x-y|.
\]
Substituting these estimates into \eqref{eq:lemma72-drift-bound-first} gives
\begin{align}\label{eq:lemma72-drift-Lipschitz}
|b_{u,\varepsilon}(x)-b_{u,\varepsilon}(y)| \leq P_1L_G(1+L_{\nabla u})|x-y|+ C_GL_{\nabla p}|x-y|= L_b|x-y|, 
\end{align}
where
\begin{equation}\label{eq:lemma72-Lb-definition}
L_b := P_1L_G(1+L_{\nabla u}) + C_GL_{\nabla p}. 
\end{equation}
Thus $b_{u,\varepsilon}$ is Lipschitz on $\overline V$, and therefore locally Lipschitz in a neighborhood of $\overline\Omega$. The drift is also bounded. Indeed,
\begin{align}\label{eq:lemma72-drift-bounded}
|b_{u,\varepsilon}(x)| = |G(x,\nabla u(x))| |\nabla p(x)| \leq C_GP_1\quad \forall x \in \overline{V}.
\end{align}

The regularized diffusion matrix can be written as
\[
D_\varepsilon(x,\xi) = F(x,\xi)I + H(x,\xi)\xi\otimes\xi,
\]
where
\begin{equation}\label{eq:lemma72-H-definition}
H(x,\xi) := (p(x)-2)r_\varepsilon(\xi)^{p(x)-4}. 
\end{equation}
To prove that $D_\varepsilon$ is Lipschitz, it is enough to prove that $F$, $H$, and the map $\xi\longmapsto\xi\otimes\xi$ are Lipschitz on the bounded set under consideration. The Lipschitz property of $F$ was established above. We now analyze $H$.\\
The function
\[
H(x,\xi) = (p(x)-2)r_\varepsilon(\xi)^{p(x)-4}
\]
is continuously differentiable on a neighborhood of $K$. Indeed, holding $\xi$ fixed gives
\begin{align}\label{eq:lemma72-H-x-derivative}
\frac{\partial H}{\partial x_k}(x,\xi) = p_{x_k}(x) r_\varepsilon(\xi)^{p(x)-4} \left[ 1+(p(x)-2)\log r_\varepsilon(\xi) \right]. 
\end{align}
For the gradient derivative, 
\begin{align}\label{eq:lemma72-H-xi-derivative}
\frac{\partial H}{\partial\xi_j}(x,\xi) = (p(x)-2)(p(x)-4) r_\varepsilon(\xi)^{p(x)-6}\xi_j. 
\end{align}
The derivative formulas above are continuous on $\mathcal U\times\mathbb R^N$ and bounded on $K$, because $r_\varepsilon\geq\varepsilon>0$ and $p$, $\nabla p$, and $\xi$ are bounded on $K$. Therefore, $H$ is Lipschitz on $K$, and hence, there is a constant $L_H>0$ such that
\begin{equation} \label{eq:lemma72-H-Lipschitz}
|H(x,\xi)-H(y,\eta)| \leq L_H \left(|x-y|+|\xi-\eta| \right)\quad \text{for every }(x,\xi),(y,\eta)\in K.
\end{equation}
Define
\begin{equation}\label{eq:lemma72-H-bound}
C_H := \sup_{(x,\xi)\in K}|H(x,\xi)|. 
\end{equation}
Let $|\xi|,|\eta|\leq M_u$. Then
\begin{align}\label{eq:lemma72-tensor-difference} 
\xi\otimes\xi-\eta\otimes\eta = \xi\otimes\xi-\eta\otimes\xi + \eta\otimes\xi-\eta\otimes\eta = (\xi-\eta)\otimes\xi + \eta\otimes(\xi-\eta).
\end{align}
For every $v,w\in\mathbb R^N$, the exact identity
\[
\|v\otimes w\|_{\mathrm{op}}=|v||w|
\]
follows from $(v\otimes w)z=v(w\cdot z)$ and Cauchy--Schwarz, with equality for $z=w/|w|$ when $w\neq0$. The same identity holds for the Frobenius norm.
Therefore,
\begin{align}\label{eq:lemma72-tensor-Lipschitz}
\|\xi\otimes\xi-\eta\otimes\eta\|_{\mathrm{op}}&\leq|\xi-\eta||\xi|+|\eta||\xi-\eta|\nonumber\\
&=(|\xi|+|\eta|)|\xi-\eta|\leq2M_u|\xi-\eta|.
\end{align}
Thus the outer-product map is Lipschitz on every bounded set. Also,
\begin{equation}\label{eq:lemma72-tensor-bound}
\|\xi\otimes\xi\|_{\mathrm{op}}=|\xi|^2\leq M_u^2.
\end{equation}
Let
\[
D_\varepsilon(x,\xi) = F(x,\xi)I + H(x,\xi)\xi\otimes\xi.
\]
For two pairs $(x,\xi)$ and $(y,\eta)$ in $K$, write
\begin{align}\label{eq:lemma72-D-difference-first}
D_\varepsilon(x,\xi)-D_\varepsilon(y,\eta) = \left[ F(x,\xi)-F(y,\eta) \right]I + H(x,\xi)\xi\otimes\xi - H(y,\eta)\eta\otimes\eta.
\end{align}
Adding and subtracting $H(y,\eta)\xi\otimes\xi$ gives
\begin{align}\label{eq:lemma72-H-tensor-decomposition}
H(x,\xi)\xi\otimes\xi - H(y,\eta)\eta\otimes\eta = \left[ H(x,\xi)-H(y,\eta) \right]\xi\otimes\xi + H(y,\eta) \left[ \xi\otimes\xi-\eta\otimes\eta \right]. 
\end{align}
Taking norms and applying the triangle inequality yields
\begin{align}\label{eq:lemma72-D-difference-bound}
\|D_\varepsilon(x,\xi)-D_\varepsilon(y,\eta)\|_{\mathrm{op}}
&\leq \|I\|_{\mathrm{op}}|F(x,\xi)-F(y,\eta)|\nonumber\\
&\quad+\|\xi\otimes\xi\|_{\mathrm{op}}|H(x,\xi)-H(y,\eta)|\nonumber\\
&\quad+|H(y,\eta)|\|\xi\otimes\xi-\eta\otimes\eta\|_{\mathrm{op}}.
\end{align}
Using \eqref{eq:lemma72-F-Lipschitz}, \eqref{eq:lemma72-H-Lipschitz}, \eqref{eq:lemma72-tensor-Lipschitz}, and \eqref{eq:lemma72-tensor-bound} gives
\begin{align}\label{eq:lemma72-D-bound-expanded}
\|D_\varepsilon(x,\xi)-D_\varepsilon(y,\eta)\|_{\mathrm{op}}&\leq \|I\|_{\mathrm{op}}L_F \left( |x-y|+|\xi-\eta| \right) + M_u^2L_H \left( |x-y|+|\xi-\eta| \right)\nonumber\\
& + 2C_HM_u|\xi-\eta|.
\end{align}
Hence there exists a constant \(L_D>0\) such that
\begin{equation}\label{eq:lemma72-D-Lipschitz-pair}
\|D_\varepsilon(x,\xi)-D_\varepsilon(y,\eta)\|_{\mathrm{op}} \leq L_D \left( |x-y|+|\xi-\eta| \right). 
\end{equation}
Thus $D_\varepsilon$ is Lipschitz with respect to $(x,\xi)$ on $K$ in \eqref{eq:lemma72-compact-set}. For the operator norm $\|I\|_{\mathrm{op}}=1$, so one may take $L_D=L_F+M_u^2L_H+2C_HM_u$.\\
Define the diffusion matrix with the profile gradient by (see \eqref{eq:intro-frozen-coefficients-a})
\[
D_{u,\varepsilon}(x) := D_\varepsilon(x,\nabla u(x)).
\]
For \(x,y\in\overline V\), apply \eqref{eq:lemma72-D-Lipschitz-pair} with $\xi=\nabla u(x)$, and $\eta=\nabla u(y)$. Then
\begin{align}\label{eq:lemma72-D-profile-Lipschitz}
\|D_{u,\varepsilon}(x)-D_{u,\varepsilon}(y)\|_{\mathrm{op}} \leq L_D \left( |x-y| + |\nabla u(x)-\nabla u(y)| \right) \leq L_D(1+L_{\nabla u})|x-y|. 
\end{align}
Therefore, $ x\longmapsto D_{u,\varepsilon}(x)$ is Lipschitz on $\overline V$. Recall the covariance field 
\begin{equation}\label{eq:thm74-covariance-definition}
a_{u,\varepsilon}(x) = \sigma_{u,\varepsilon}(x) \sigma_{u,\varepsilon}(x)^T = 2D_\varepsilon(x,\nabla u(x)). 
\end{equation}
Equivalently,
\[
a_{u,\varepsilon}(x) = 2D_{u,\varepsilon}(x).
\]
Consequently,
\begin{align}\label{eq:lemma72-a-Lipschitz}
\|a_{u,\varepsilon}(x)-a_{u,\varepsilon}(y)\|_{\mathrm{op}} = 2\|D_{u,\varepsilon}(x)-D_{u,\varepsilon}(y)\|_{\mathrm{op}} \leq 2L_D(1+L_{\nabla u})|x-y|. 
\end{align}
Thus \(a_{u,\varepsilon}\) is Lipschitz on \(\overline V\).\\
Set $\widehat M_u=\max\{1,M_u\}>0$. The exponent bounds of Assumption \ref{ass:classical-framework} hold throughout $\mathcal{U}$, so the pointwise estimates of Lemma \ref{thm:reg-uniform} apply on
$\overline V$ with $M=\widehat M_u$. Define
\[
\alpha_u:=\alpha_{\varepsilon,\widehat M_u},\quad \beta_u:=\beta_{\varepsilon,\widehat M_u}.
\]
Then $0<\alpha_u\leq \beta_u<\infty$ and
\[
\alpha_u|\zeta|^2\leq \zeta^TD_{u,\varepsilon}(x)\zeta \leq\beta_u|\zeta|^2
\quad \forall\, (x, \zeta) \in \overline V\times\mathbb R^N.
\]
This choice also covers a constant profile, for which $M_u=0$. Since $a_{u,\varepsilon}(x) = 2D_{u,\varepsilon}(x)$, we obtain
\begin{equation}\label{eq:lemma72-a-ellipticity}
2\alpha_u|\zeta|^2 \leq \zeta^Ta_{u,\varepsilon}(x)\zeta \leq 2\beta_u|\zeta|^2.
\end{equation}
We now show that the quadratic-form bounds in \eqref{eq:lemma72-a-ellipticity} are equivalent to the matrix inequalities
\begin{equation}\label{eq:lemma72-a-matrix-bounds}
2\alpha_u I \leq a_{u,\varepsilon}(x) \leq 2\beta_u I.
\end{equation}
For real symmetric matrices $A$ and $B$, the Loewner order relation $ A\leq B$ means that the matrix $ B-A$ is positive semidefinite. Since $|\zeta|^2 = \zeta^TI\zeta$, we have
\[
2\alpha_u|\zeta|^2 = \zeta^T(2\alpha_u I)\zeta \quad \text{and} \quad 2\beta_u|\zeta|^2 = \zeta^T(2\beta_u I)\zeta.
\]
Therefore, \eqref{eq:lemma72-a-ellipticity} can be rewritten as
\begin{equation}\label{eq:quadratic-form-matrix-version}
\zeta^T(2\alpha_u I)\zeta \leq \zeta^Ta_{u,\varepsilon}(x)\zeta \leq \zeta^T(2\beta_u I)\zeta
\end{equation}
for every $\zeta\in\mathbb R^N$. The lower inequality in \eqref{eq:quadratic-form-matrix-version} gives
\begin{align*}
0 \leq \zeta^Ta_{u,\varepsilon}(x)\zeta - \zeta^T(2\alpha_u I)\zeta = \zeta^T \left( a_{u,\varepsilon}(x)-2\alpha_u I \right) \zeta.
\end{align*}
Since this holds for every \(\zeta\), the matrix $a_{u,\varepsilon}(x)-2\alpha_u I$ is positive semidefinite. Hence
\[
2\alpha_u I \leq a_{u,\varepsilon}(x).
\]
Similarly, the upper inequality gives
\begin{align*}
0 \leq \zeta^T(2\beta_u I)\zeta -\zeta^Ta_{u,\varepsilon}(x)\zeta = \zeta^T \left(2\beta_u I-a_{u,\varepsilon}(x) \right ) \zeta.
\end{align*}
Thus $2\beta_u I-a_{u,\varepsilon}(x)$ is positive semidefinite, and therefore
\[
a_{u,\varepsilon}(x) \leq 2\beta_u I.
\]
Combining the two inequalities gives \eqref{eq:lemma72-a-matrix-bounds}. Note that the symmetry required for the Loewner order follows from $ a_{u,\varepsilon}(x) = 2D_\varepsilon(x,\nabla u(x))$ and the fact that $D_\varepsilon$ is symmetric.\\
Finally, let $\lambda$ be an eigenvalue of $a_{u,\varepsilon}(x)$, with nonzero eigenvector $v$. Then
\[
a_{u,\varepsilon}(x)v=\lambda v.
\]
Letting $\zeta=v$ in \eqref{eq:lemma72-a-ellipticity} gives
\[
2\alpha_u|v|^2 \leq v^Ta_{u,\varepsilon}(x)v \leq 2\beta_u|v|^2.
\]
Since
\[
v^Ta_{u,\varepsilon}(x)v = \lambda|v|^2,
\]
we obtain
\[
2\alpha_u|v|^2 \leq \lambda|v|^2 \leq 2\beta_u|v|^2.
\]
Dividing by $|v|^2>0$ gives
\[
2\alpha_u\leq\lambda\leq2\beta_u.
\]
Thus every eigenvalue of $a_{u,\varepsilon}(x)$ belongs to $[2\alpha_u,2\beta_u]\subset(0,\infty)$. In particular, $a_{u,\varepsilon}(x)$ is symmetric and strictly positive definite. Therefore $a_{u,\varepsilon}(x)$ can be orthogonally diagonalized; that is, there is an orthogonal matrix $Q(x)$ and positive eigenvalues $\lambda_1(x),\ldots,\lambda_N(x)$ such that
\[
a_{u,\varepsilon}(x)= Q(x) \operatorname{diag} \left(\lambda_1(x),\ldots,\lambda_N(x) \right) Q(x)^T.
\]
Then the principal symmetric square root is defined by
\[
a_{u,\varepsilon}(x)^{1/2} = Q(x) \operatorname{diag} \left(\sqrt{\lambda_1(x)},\ldots,\sqrt{\lambda_N(x)} \right ) Q(x)^T.
\]
This matrix is symmetric positive definite and satisfies
\[
a_{u,\varepsilon}(x)^{1/2} a_{u,\varepsilon}(x)^{1/2} = a_{u,\varepsilon}(x).
\]
We denote the principal square root by
\begin{equation}\label{eq:lemma72-a-b}
\sigma_{u,\varepsilon}(x) = a_{u,\varepsilon}(x)^{1/2},
\end{equation}
which is the diffusion coefficient of the SDE.\\
By \eqref{eq:lemma72-a-matrix-bounds}, for $x,y\in\overline V$ both $a_{u,\varepsilon}(x)$ and $a_{u,\varepsilon}(y)$ are bounded below by $2\alpha_u I$. We apply Theorem X.3.8 of \cite{Bhatia1997} (by letting $r=\frac{1}{2}$  and $ a=2\alpha_u $).
Then,
\begin{align} \label{eq:lemma72-resolvent-square-root-estimate}
\left\| a_{u,\varepsilon}(x)^{1/2} - a_{u,\varepsilon}(y)^{1/2} \right\|_{\mathrm{op}} &\leq \frac{1}{2\sqrt{2\alpha_u}} \left\| a_{u,\varepsilon}(x) - a_{u,\varepsilon}(y) \right\|_{\mathrm{op}}.
\end{align}
Since $a_{u,\varepsilon}$ is locally Lipschitz, using the Lipschitz estimate \eqref{eq:lemma72-a-Lipschitz}, we obtain
\begin{align}\label{eq:lemma72-resolvent-sigma-Lipschitz}
\left\| \sigma_{u,\varepsilon}(x) - \sigma_{u,\varepsilon}(y) \right\|_{\mathrm{op}} = \left\| a_{u,\varepsilon}(x)^{1/2} - a_{u,\varepsilon}(y)^{1/2} \right\|_{\mathrm{op}}
\leq \frac{ L_D(1+L_{\nabla u}) }{ \sqrt{2\alpha_u} } |x-y|.
\end{align}
Thus $\sigma_{u,\varepsilon}$ is Lipschitz on $\overline V$, and therefore locally Lipschitz on the neighborhood $V$ of $\overline\Omega$. One may take
$L_\sigma=L_D(1+L_{\nabla u})/\sqrt{2\alpha_u}$.\\

By \eqref{eq:lemma72-a-matrix-bounds}, $a_{u,\varepsilon}(x) \leq 2\beta_u I$. Therefore, every eigenvalue of $\sigma_{u,\varepsilon}(x)= a_{u,\varepsilon}(x)^{1/2}$ is at most $\sqrt{2\beta_u}$. Consequently,
\begin{equation}\label{eq:lemma72-sigma-bounded}
\|\sigma_{u,\varepsilon}(x)\|_{\mathrm{op}} \leq \sqrt{2\beta_u} \quad \text{for every }x\in\overline V. 
\end{equation}
Thus \(\sigma_{u,\varepsilon}\) is bounded on $\overline V$, and in particular on $\overline\Omega$. The same spectral bounds give $\|D_{u,\varepsilon}(x)\|_{\mathrm{op}}\leq \beta_u$ on $\overline V$. In the Frobenius norm, $\|\sigma_{u,\varepsilon}(x)\|_{\mathrm F}\leq \sqrt{2N\beta_u}$, and its Lipschitz constant is at most $\sqrt N L_\sigma$.\\
We have shown that there are constants $L_b,L_\sigma<\infty$ such that
\[
|b_{u,\varepsilon}(x)-b_{u,\varepsilon}(y)| \leq L_b|x-y|
\]
and
\[
\|\sigma_{u,\varepsilon}(x) - \sigma_{u,\varepsilon}(y)\|_{\mathrm{op}} \leq L_\sigma|x-y|
\]
for all $x,y\in\overline V$. Since $V$ is a neighborhood of $\overline\Omega$, both coefficient fields are locally Lipschitz in a neighborhood of $\overline\Omega$.
We have also shown that
\begin{equation}\label{difd-sigma-bounded}
\sup_{x\in\overline\Omega} |b_{u,\varepsilon}(x)| <\infty \quad \text{and} \quad \sup_{x\in\overline\Omega} \|\sigma_{u,\varepsilon}(x)\|_{\mathrm{op}} <\infty.
\end{equation}
Thus all three fields $D_{u,\varepsilon}$, $b_{u,\varepsilon}$, and$\sigma_{u,\varepsilon}$ are bounded on $\overline\Omega$ and Lipschitz on $\overline V$, and hence,  they are locally Lipschitz on
$\mathcal U$. In \eqref{eq:sigma-spectral}, the eigenvalues are evaluated at $(x,\nabla u(x))$ and $P_{\nabla u}=P_{\nabla u(x)}$. The formula follows by taking square roots in the orthogonal spectral decomposition from Lemma~\ref{lem:reg-spectrum}. For $N=1$, $I-P_{\nabla u}=0$, so the expression has only its radial term. Formula \eqref{eq:sigma-zero} follows directly from \eqref{eq:D-eps-zero}.
\end{proof}
\begin{corollary}\label{cor:linear-growth}
Under Assumption \ref{ass:classical-framework}, there is $C>0$ such that
\[
 \abs{b_{u,\varepsilon}(x)}+\norm{\sigma_{u,\varepsilon}(x)}_{\mathrm{op}}\leq C(1+\abs{x})\quad \forall x\in\overline\Omega.
\]
\end{corollary}
\begin{proof}
Boundedness follows from Lemma \ref{lem:lipschitz-coeff}.
Set
\[
 B:=\sup_{z\in\overline\Omega}\left(|b_{u,\varepsilon}(z)|+\norm{\sigma_{u,\varepsilon}(z)}_{\mathrm{op}}\right)<\infty.
\]
Since $1\leq 1+|x|$ for every $x\in\overline\Omega$,
\[
 |b_{u,\varepsilon}(x)|+\norm{\sigma_{u,\varepsilon}(x)}_{\mathrm{op}} \leq B\leq B(1+|x|).
\]
Thus the stated estimate holds with $C=\max\{1,B\}>0$. The estimate concerns the coefficients on $\overline\Omega$; the global extensions used in the next proof are constructed separately.
\end{proof}

We now are ready to present the main conclusions of the paper.

\section{The profile-frozen diffusion and its stochastic identities}\label{sec:main-results}

In Theorem \ref{thm:sde-existence} and Corollary \ref{cor:sde-stability}, every unsubscripted matrix norm is the Frobenius norm.  By \eqref{eq:matrix-norm-equivalence}, the coefficient bounds and Lipschitz constants established in operator norm also hold in this norm after multiplication by at most $\sqrt N$.  The constants $L_\sigma$ and $S_0$ below are chosen in the Frobenius norm from the outset. This fixes a single norm for the cutoff estimates, It\^o isometry, and martingale inequalities.

\subsection{Strong construction and stability}

\begin{theorem}\label{thm:sde-existence}
Under Assumption \ref{ass:classical-framework}, for every $x\in\Omega$ there is a unique strong solution of \eqref{eq:intro-profile-sde} up to $\tau_\Omega$. Pathwise uniqueness holds up to exit, and no finite-time explosion can occur before $\tau_\Omega$.
\end{theorem}

\begin{proof} We prove strong existence, pathwise uniqueness, and absence of finite-time explosion before the exit time. By Lemma \ref{lem:lipschitz-coeff}, the maps $b_{u,\varepsilon}$ and $\sigma_{u,\varepsilon}$ are locally Lipschitz in a neighborhood of $\overline\Omega$, and they are bounded on $\overline\Omega$. We organize the proof into several steps.\\

By Assumption \ref{ass:framework}, the functions $p$ and $u$, and therefore the coefficient fields $b_{u,\varepsilon}$ and $\sigma_{u,\varepsilon}$, are defined on an open neighborhood $\mathcal{U}$ of $\overline\Omega$. Because $\overline\Omega$ is compact and $\mathcal{U}$ is open, there exist bounded open sets $V_1$ and $V_2$ such that
\begin{equation}\label{eq:thm73-nested-neighborhoods}
\overline\Omega \subset V_1, \quad \overline{V_1} \subset V_2, \quad \overline{V_2} \subset \mathcal{U}. 
\end{equation}
Since $b_{u,\varepsilon}$ and $\sigma_{u,\varepsilon}$ are locally Lipschitz on $\mathcal{U}$, compactness of $\overline{V_2}$ implies that there are finite constants $L_b,L_\sigma$ such that \begin{equation}\label{eq:thm73-local-b-Lipschitz}
|b_{u,\varepsilon}(x)-b_{u,\varepsilon}(y)| \leq L_b|x-y| 
\end{equation}
and
\begin{equation}\label{eq:thm73-local-sigma-Lipschitz}
\|\sigma_{u,\varepsilon}(x)-\sigma_{u,\varepsilon}(y)\| \leq L_\sigma|x-y| 
\end{equation}
for all $x,y\in\overline{V_2}$. Moreover, continuity on the compact set $\overline{V_2}$ gives finite bounds
\begin{equation}\label{eq:thm73-b-bound}
B_0 := \sup_{x\in\overline{V_2}}|b_{u,\varepsilon}(x)| <\infty 
\end{equation}
and
\begin{equation}\label{eq:thm73-sigma-bound}
S_0 := \sup_{x\in\overline{V_2}}\|\sigma_{u,\varepsilon}(x)\| <\infty. 
\end{equation} 
Choose a smooth cutoff function $\chi\in C_c^\infty(V_2)$ such that
\begin{equation}\label{eq:thm73-cutoff-range}
0\leq\chi\leq1 \quad\text{on }\mathbb R^N 
\end{equation}
and
\begin{equation}\label{eq:thm73-cutoff-one}
\chi(x)=1 \quad \text{for every }x\in\overline{V_1}. 
\end{equation}
Define the extended drift by
\begin{equation}\label{eq:thm73-extended-drift}
\widetilde b_{u,\varepsilon}(x) = 
\begin{cases} \chi(x)b_{u,\varepsilon}(x),&x\in V_2,\\ 0,&x\notin V_2, \end{cases} 
\end{equation}
and define the extended diffusion coefficient by
\begin{equation}\label{eq:thm73-extended-diffusion}
\widetilde\sigma_{u,\varepsilon}(x) =
\begin{cases}
\chi(x)\sigma_{u,\varepsilon}(x),&x\in V_2,\\ 0,&x\notin V_2.
\end{cases} 
\end{equation}
Because $\chi$ has compact support contained strictly inside $V_2$, the functions $\chi b_{u,\varepsilon}$ and $\chi\sigma_{u,\varepsilon}$ vanish in a neighborhood of $\partial V_2$. Their zero extensions are therefore continuous and Lipschitz across $\partial V_2$. We verify the Lipschitz estimate inside $V_2$. For $x,y\in V_2$,
\begin{align} \label{eq:thm73-extended-b-difference}
\widetilde b_{u,\varepsilon}(x)-\widetilde b_{u,\varepsilon}(y) &= \chi(x)b_{u,\varepsilon}(x)-\chi(y)b_{u,\varepsilon}(y) \nonumber\\ &= \chi(x)\bigl(b_{u,\varepsilon}(x)-b_{u,\varepsilon}(y)\bigr) + \bigl(\chi(x)-\chi(y)\bigr)b_{u,\varepsilon}(y).
\end{align}
Taking norms gives
\begin{align}\label{eq:thm73-extended-b-bound-first}
|\widetilde b_{u,\varepsilon}(x)-\widetilde b_{u,\varepsilon}(y)| \leq |\chi(x)| |b_{u,\varepsilon}(x)-b_{u,\varepsilon}(y)| + |\chi(x)-\chi(y)| |b_{u,\varepsilon}(y)|. 
\end{align}
Since
\[
|\chi(x)|\leq1, \quad |b_{u,\varepsilon}(x)-b_{u,\varepsilon}(y)|\leq L_b|x-y|, \quad \text{and} \quad|b_{u,\varepsilon}(y)|\leq B_0,
\]
we obtain
\begin{align}\label{eq:thm73-extended-b-bound-second}
|\widetilde b_{u,\varepsilon}(x)-\widetilde b_{u,\varepsilon}(y)| &\leq L_b|x-y| + B_0|\chi(x)-\chi(y)|. 
\end{align}
Since $\chi$ is smooth with compact support, its gradient is bounded. Define
\[
L_\chi := \sup_{z\in\mathbb R^N}|\nabla\chi(z)|.
\]
Then the MVT gives
\[
|\chi(x)-\chi(y)| \leq L_\chi|x-y|.
\]
Therefore,
\begin{equation}\label{eq:thm73-extended-b-Lipschitz}
|\widetilde b_{u,\varepsilon}(x)-\widetilde b_{u,\varepsilon}(y)| \leq \left( L_b+B_0L_\chi \right)|x-y|. 
\end{equation}
The same calculation gives
\begin{equation}\label{eq:thm73-extended-sigma-Lipschitz}
\|\widetilde\sigma_{u,\varepsilon}(x)-\widetilde\sigma_{u,\varepsilon}(y)\| \leq \left( L_\sigma+S_0L_\chi \right)|x-y|. 
\end{equation}
For the remaining pairs, let $x\in V_2$ and $y\notin V_2$. Since $\chi(y)=0$, the globally defined smooth cutoff satisfies
\[
|\widetilde b_{u,\varepsilon}(x)-\widetilde b_{u,\varepsilon}(y)|=|\chi(x)b_{u,\varepsilon}(x)|\leq B_0L_\chi|x-y|,
\]
and the corresponding diffusion bound is $S_0L_\chi|x-y|$. If both points lie outside $V_2$, both differences vanish. Combining these bounds with the estimates inside $V_2$ proves that both zero
extensions are globally Lipschitz. Thus there is a constant $L>0$ such that
\begin{equation}\label{eq:thm73-global-Lipschitz}
|\widetilde b_{u,\varepsilon}(x)-\widetilde b_{u,\varepsilon}(y)| + \|\widetilde\sigma_{u,\varepsilon}(x)-\widetilde\sigma_{u,\varepsilon}(y)\| \leq L|x-y| 
\end{equation}
for all $x,y\in\mathbb R^N$. The cutoff also gives global boundedness:
\begin{equation}\label{eq:thm73-global-b-bound}
|\widetilde b_{u,\varepsilon}(x)| \leq B_0 
\end{equation}
and
\begin{equation}\label{eq:thm73-global-sigma-bound}
\|\widetilde\sigma_{u,\varepsilon}(x)\| \leq S_0 
\end{equation}
for every $x\in\mathbb R^N$. In particular, the extended coefficients satisfy the standard linear-growth estimate
\begin{equation} \label{eq:thm73-linear-growth}
|\widetilde b_{u,\varepsilon}(x)|^2 + \|\widetilde\sigma_{u,\varepsilon}(x)\|^2 \leq K_0(1+|x|^2)
\end{equation}
for some constant $K_0>0$. Since the coefficients are bounded, one may take, for example, $K_0:=B_0^2+S_0^2$. Finally, because $\chi=1$ on $\overline{V_1}$, we have
\begin{equation}\label{eq:thm73-coefficients-agree}
\widetilde b_{u,\varepsilon}(x)=b_{u,\varepsilon}(x), \quad \widetilde\sigma_{u,\varepsilon}(x)=\sigma_{u,\varepsilon}(x) \quad \text{for every }x\in\overline{V_1}. 
\end{equation}
Since $\overline\Omega\subset V_1$, the extended and original coefficients agree on a neighborhood of $\overline\Omega$.\\
Consider
\begin{equation}\label{eq:thm73-extended-SDE}
\mathrm d\widetilde X_t = \widetilde b_{u,\varepsilon}(\widetilde X_t)\,\mathrm dt + \widetilde\sigma_{u,\varepsilon}(\widetilde X_t)\,\mathrm dW_t, \quad \widetilde X_0=x. 
\end{equation}
In integral form, this means
\begin{equation}\label{eq:thm73-extended-integral}
\widetilde X_t = x + \int_0^t \widetilde b_{u,\varepsilon}(\widetilde X_s)\,\mathrm ds + \int_0^t \widetilde\sigma_{u,\varepsilon}(\widetilde X_s)\,\mathrm dW_s. 
\end{equation}
We now construct a solution by Picard iteration.\\
Set
\[
X_t^{(0)}=x \quad \text{for every }t\geq0.
\]
Recursively define
\begin{equation}\label{eq:thm73-Picard-iteration}
X_t^{(n+1)} = x + \int_0^t \widetilde b_{u,\varepsilon}(X_s^{(n)}) ds + \int_0^t \widetilde\sigma_{u,\varepsilon}(X_s^{(n)}) dW_s. 
\end{equation}
Each process $X^{(n)}$ is adapted and has continuous sample paths. Inductively, each previous iterate is progressively measurable, and the globally bounded coefficients make its stochastic integrand square integrable on every finite interval. Indeed, the deterministic integral is continuous, and the It\^o integral has a continuous modification. We estimate the difference between two consecutive iterates. Define
\[
\Delta_t^{(n)} = X_t^{(n+1)}-X_t^{(n)}.
\]
Then, for $n\geq1$, 
\begin{align}\label{eq:thm73-Picard-difference}
\Delta_t^{(n)} = \int_0^t \left[ \widetilde b_{u,\varepsilon}(X_s^{(n)}) - \widetilde b_{u,\varepsilon}(X_s^{(n-1)}) \right]ds + \int_0^t \left[ \widetilde\sigma_{u,\varepsilon}(X_s^{(n)}) - \widetilde\sigma_{u,\varepsilon}(X_s^{(n-1)}) \right]dW_s. 
\end{align}
Applying
\begin{equation}\label{eq:binomial-2}
|a+b|^2\leq2|a|^2+2|b|^2
\end{equation}
to \eqref{eq:thm73-Picard-difference} gives
\begin{align}\label{eq:thm73-Picard-sup}
\sup_{0\leq r\leq t} |\Delta_r^{(n)}|^2 &\leq 2\sup_{0\leq r\leq t} \left| \int_0^r \left[ \widetilde b_{u,\varepsilon}(X_s^{(n)}) - \widetilde b_{u,\varepsilon}(X_s^{(n-1)}) \right]ds \right|^2 \nonumber\\ &\quad+ 2\sup_{0\leq r\leq t} \left| \int_0^r \left[ \widetilde\sigma_{u,\varepsilon}(X_s^{(n)}) - \widetilde\sigma_{u,\varepsilon}(X_s^{(n-1)}) \right]dW_s \right|^2. 
\end{align}
Taking expectations in \eqref{eq:thm73-Picard-sup}, then applying the Cauchy--Schwarz inequality, the Burkholder--Davis--Gundy (BDG) \cite{bg1970,davis1970} inequality (which gives a constant $C_{\mathrm{BDG}}>0$), and using global Lipschitz continuity, we obtain
\begin{align}\label{eq:thm73-Picard-estimate-first}
\mathbb{E}\left[\sup_{0\leq r\leq t} |\Delta_r^{(n)}|^2 \right] \leq 2tL^2 \int_0^t \mathbb{E}\left| X_s^{(n)}-X_s^{(n-1)} \right|^2 ds + 2C_{\mathrm{BDG}}L^2 \int_0^t \mathbb{E}\left| X_s^{(n)}-X_s^{(n-1)} \right|^2 ds. 
\end{align}
Since
\[
|X_s^{(n)}-X_s^{(n-1)}|^2 \leq \sup_{0\leq q\leq s} |X_q^{(n)}-X_q^{(n-1)}|^2,
\]
we obtain
\begin{align}\label{eq:thm73-Picard-recursive}
\mathbb{E}\left[ \sup_{0\leq r\leq t} |\Delta_r^{(n)}|^2 \right] &\leq C_T \int_0^t \mathbb{E}\left[\sup_{0\leq q\leq s} |\Delta_q^{(n-1)}|^2 \right] ds 
\end{align}
for \(t\leq T\), where $C_T := 2L^2(T+C_{\mathrm{BDG}})$. Repeated application of \eqref{eq:thm73-Picard-recursive} gives a factorial estimate of the form
\begin{equation}\label{eq:thm73-Picard-factorial}
\mathbb{E}\left[ \sup_{0\leq t\leq T} |\Delta_t^{(n)}|^2 \right] \leq \frac{C_0(C_TT)^n}{n!} 
\end{equation}
for a finite constant $C_0$ depending on the first Picard difference. Because
\[
\sum_{n=0}^\infty \frac{(C_TT)^n}{n!} = e^{C_TT} <\infty,
\]
the Picard sequence is Cauchy in in the Banach space $L^2(\Xi;C([0,T];\mathbb R^N))$, where $\Xi$ is the probability space fixed in Section \ref{sec:framework}. Therefore, there exists a continuous adapted version $\widetilde X$ agrees with the Banach space limit, so
\begin{equation} \label{eq:thm73-Picard-convergence}
\mathbb{E}\left[ \sup_{0\leq t\leq T} |X_t^{(n)}-\widetilde X_t|^2 \right] \rightarrow0.
\end{equation}

Global Lipschitz continuity gives
\[
|\widetilde b_{u,\varepsilon}(X_s^{(n)}) - \widetilde b_{u,\varepsilon}(\widetilde X_s)| \leq L|X_s^{(n)}-\widetilde X_s|.
\]
Therefore,
\begin{align}\label{eq:thm73-drift-limit}
\mathbb{E}\left| \int_0^t \left[\widetilde b_{u,\varepsilon}(X_s^{(n)}) - \widetilde b_{u,\varepsilon}(\widetilde X_s) \right]ds \right|^2 &\leq tL^2 \int_0^t \mathbb{E}|X_s^{(n)}-\widetilde X_s|^2 ds \rightarrow0. 
\end{align}
Similarly, the It\^o isometry gives
\begin{align} \label{eq:thm73-diffusion-limit}
\mathbb{E}\left| \int_0^t \left[ \widetilde\sigma_{u,\varepsilon}(X_s^{(n)}) - \widetilde\sigma_{u,\varepsilon}(\widetilde X_s) \right]dW_s \right|^2 &= \mathbb{E}\int_0^t \left\| \widetilde\sigma_{u,\varepsilon}(X_s^{(n)})- \widetilde\sigma_{u,\varepsilon}(\widetilde X_s) \right\|^2 ds\leq L^2 \int_0^t \mathbb{E}|X_s^{(n)}-\widetilde X_s|^2ds \nonumber\\ &\rightarrow0.
\end{align}
Passing to the limit in \eqref{eq:thm73-Picard-iteration} therefore yields
\[
\widetilde X_t = x + \int_0^t \widetilde b_{u,\varepsilon}(\widetilde X_s)ds + \int_0^t \widetilde\sigma_{u,\varepsilon}(\widetilde X_s)dW_s.
\]
The limit equation holds almost surely at every fixed time. The limit is progressively measurable because it is continuous and adapted, so its stochastic integral is well defined.  Thus $\widetilde X$ is a strong solution of the globally extended SDE on $[0,T]$.

Let $\widetilde X$ and $\widetilde Y$ be two strong solutions of the extended SDE, defined on the same filtered probability space, driven by the same Brownian motion, and satisfying
\[
\widetilde X_0=\widetilde Y_0=x.
\]
Define
\[
Z_t=\widetilde X_t-\widetilde Y_t.
\]
Subtracting the two integral equations gives
\begin{align}\label{eq:thm73-uniqueness-difference}
Z_t = \int_0^t \left[ \widetilde b_{u,\varepsilon}(\widetilde X_s) - \widetilde b_{u,\varepsilon}(\widetilde Y_s) \right] ds + \int_0^t \left[ \widetilde\sigma_{u,\varepsilon}(\widetilde X_s) - \widetilde\sigma_{u,\varepsilon}(\widetilde Y_s) \right]dW_s. 
\end{align}
Using the same estimates as in the Picard argument, we obtain
\begin{equation}\label{eq:thm73-uniqueness-Gronwall}
\mathbb{E}\left[ \sup_{0\leq r\leq t} |Z_r|^2 \right] \leq C_T \int_0^t \mathbb{E}\left[ \sup_{0\leq q\leq s} |Z_q|^2 \right]ds. 
\end{equation}
Define
\[
F(t) := \mathbb{E}\left[ \sup_{0\leq r\leq t} |Z_r|^2 \right].
\]
Then
\[
F(t)\leq C_T\int_0^tF(s)ds \quad \text{and}\quad F(0)=0.
\]
Gronwall's inequality gives
\[
F(t)=0 \quad \text{for every }t\in[0,T].
\]
Therefore,
\[
\mathbb{P}\left( \widetilde X_t=\widetilde Y_t \text{ for every }t\in[0,T] \right)=1.
\]
Since $T>0$ is arbitrary, pathwise uniqueness holds globally for the extended SDE.\\

The boundedness of the extended coefficients also prevents finite-time explosion. From the integral equation,
\[
\widetilde X_t = x + \int_0^t\widetilde b_{u,\varepsilon}(\widetilde X_s)ds + \int_0^t\widetilde\sigma_{u,\varepsilon}(\widetilde X_s)dW_s.
\]
Using \eqref{eq:binomial-2} for three terms, we obtain
\begin{align}\label{eq:thm73-moment-start}
\sup_{0\leq r\leq t}|\widetilde X_r|^2 \leq 3|x|^2 + 3\sup_{0\leq r\leq t} \left| \int_0^r\widetilde b_{u,\varepsilon}(\widetilde X_s)ds \right|^2 + 3\sup_{0\leq r\leq t} \left| \int_0^r\widetilde\sigma_{u,\varepsilon}(\widetilde X_s)dW_s \right|^2. 
\end{align}
Using the boundedness estimates
\[
|\widetilde b_{u,\varepsilon}|\leq B_0, \quad \|\widetilde\sigma_{u,\varepsilon}\|\leq S_0,
\]
the deterministic integral satisfies
\[
\sup_{0\leq r\leq t} \left| \int_0^r\widetilde b_{u,\varepsilon}(\widetilde X_s)ds \right|^2 \leq B_0^2t^2.
\]
The BDG inequality gives
\begin{align}\label{eq:thm73-moment-martingale}
\mathbb{E}\left[ \sup_{0\leq r\leq t} \left| \int_0^r\widetilde\sigma_{u,\varepsilon}(\widetilde X_s)dW_s \right|^2 \right] \leq C_{\mathrm{BDG}} \mathbb{E}\int_0^t \|\widetilde\sigma_{u,\varepsilon}(\widetilde X_s)\|^2 ds \leq C_{\mathrm{BDG}}S_0^2t. 
\end{align}
Therefore,
\begin{equation} \label{eq:thm73-moment-final}
\mathbb{E}\left[\sup_{0\leq r\leq t} |\widetilde X_r|^2 \right] \leq 3|x|^2 + 3B_0^2t^2 + 3C_{\mathrm{BDG}}S_0^2t <\infty.
\end{equation}
In particular, the extended process cannot reach infinite Euclidean norm in finite time. More explicitly, set $\rho_R=\inf\{t\geq0:|\widetilde X_t|\geq R\}$.
For $R>|x|$, Markov's inequality and \eqref{eq:thm73-moment-final} imply
\[
 \mathbb{P}(\rho_R\leq T) \leq \frac{3|x|^2+3B_0^2T^2+3C_{\mathrm{BDG}}S_0^2T}{R^2} \rightarrow 0.
\]
Hence $\rho_R\uparrow\infty$ almost surely, which is the nonexplosion conclusion for the globally extended SDE.\\

Using the global continuous process $\widetilde X$, define
\begin{equation}\label{eq:thm73-exit-time}
\tau_\Omega = \inf\left\{ t>0: \widetilde X_t\notin\Omega \right\}. 
\end{equation}
Because $\Omega$ is open and $\widetilde X$ is adapted with continuous paths, $\tau_\Omega$ is a stopping time. For every $t<\tau_\Omega$, we have $\widetilde X_t\in\Omega$. Since
$\Omega\subset V_1$ and the extended coefficients agree with the original coefficients on $V_1$, it follows that
\[
\widetilde b_{u,\varepsilon}(\widetilde X_t) = b_{u,\varepsilon}(\widetilde X_t) \quad \text{and} \quad \widetilde\sigma_{u,\varepsilon}(\widetilde X_t) = \sigma_{u,\varepsilon}(\widetilde X_t)
\]
for every $t<\tau_\Omega$. Define the stopped process
\begin{equation}\label{eq:thm73-stopped-process}
X_t = \widetilde X_{t\wedge\tau_\Omega}. 
\end{equation}
For $t<\tau_\Omega$, $X_t=\widetilde X_t$, and therefore
\[
X_t = x + \int_0^t b_{u,\varepsilon}(X_s)ds + \int_0^t \sigma_{u,\varepsilon}(X_s) dW_s.
\]
Equivalently, in stopped form,
\begin{align} \label{eq:thm73-stopped-integral-equation}
X_t &= x + \int_0^{t\wedge\tau_\Omega} b_{u,\varepsilon}(X_s)ds + \int_0^{t\wedge\tau_\Omega} \sigma_{u,\varepsilon}(X_s)dW_s.
\end{align}
Thus $X$ is a strong solution of the original regularized SDE up to the exit time $\tau_\Omega$.\\

Let $X$ and $Y$ be two strong solutions of the profile-frozen SDE \eqref{eq:intro-profile-sde}, defined on the same filtered probability space, driven by the same Brownian motion, and starting from the same point $x\in\Omega$. Let
\[
\tau_\Omega^X = \inf\{t>0:X_t\notin\Omega\} \quad \text{and} \quad \tau_\Omega^Y = \inf\{t>0:Y_t\notin\Omega\}.
\]
Define the common stopping time by $\tau := \tau_\Omega^X\wedge\tau_\Omega^Y$. Before time $\tau$, both processes remain in $\Omega$, where the original coefficients are Lipschitz. The difference satisfies
\begin{align}\label{eq:thm73-stopped-uniqueness}
X_{t\wedge\tau}-Y_{t\wedge\tau} = \int_0^{t\wedge\tau} \left[ b_{u,\varepsilon}(X_s)-b_{u,\varepsilon}(Y_s) \right]ds + \int_0^{t\wedge\tau}
\left[\sigma_{u,\varepsilon}(X_s)-\sigma_{u,\varepsilon}(Y_s) \right]dW_s. 
\end{align}
The same BDG, Cauchy--Schwarz, and Gronwall argument used above gives
\[
\mathbb{E}\left[\sup_{0\leq s\leq t } |X_{s\wedge\tau}-Y_{s\wedge\tau}|^2 \right] = 0.
\]
Therefore, $X_{t\wedge\tau} = Y_{t\wedge\tau}$ almost surely for every $t\geq0$. Since the paths coincide up to the first exit of either process, continuity implies that the exit times coincide; that is, $\tau_\Omega^X=\tau_\Omega^Y$ almost surely. Hence $X_t=Y_t$ for every $t<\tau_\Omega$ almost surely. This proves pathwise uniqueness up to the exit time.\\
A finite-time explosion before exit would mean that there is a finite random time $\zeta<\tau_\Omega$ such that
\[
|X_t|\longrightarrow\infty \quad \text{as }t\uparrow\zeta.
\]
This cannot occur for two independent reasons. First, the process $X$ is the restriction of the global, non-explosive process $\widetilde X$. The process $\widetilde X_t$ is finite for every finite $t$ almost surely. Second, for every $t<\tau_\Omega$, $X_t\in\Omega$.  Because $\Omega$ is bounded, there is a number $R_\Omega<\infty$ such that
\[
|z|\leq R_\Omega \quad \text{ for every }z\in\Omega.
\]
Therefore,
\[
|X_t|\leq R_\Omega \quad \text{for every }t <\tau_\Omega.
\]
Thus $|X_t|$ cannot tend to infinity while the process remains in $\Omega$. Consequently, no finite-time explosion occurs before $\tau_\Omega$. We have proved that, for every $x\in\Omega$, the regularized SDE admits a strong solution up to the first exit time, that the solution is pathwise unique up to that time, and that no finite-time explosion can occur before exit. This completes the proof.
\end{proof}

As a conclusion, we have the following initial-point stability result for $\widetilde X$.
\begin{corollary}\label{cor:sde-stability}
Under Assumption \ref{ass:classical-framework}, let $\widetilde X^x$ and $\widetilde X^y$ solve the same globally extended SDE \eqref{eq:thm73-extended-SDE}, with one fixed choice of cutoff and extended coefficients, with initial values $x$ and $y$ and the same Brownian motion.  For every $T>0$ there is $C_T>0$ such that
\[
\mathbb{E}\left[\sup_{0\leq t\leq T}\abs{\widetilde X_t^x-\widetilde X_t^y}^2\right] \leq C_T\abs{x-y}^2.
\]
\end{corollary}
\begin{proof}
Subtract the two integral equations:
\begin{align*}
\widetilde X_t^x-\widetilde X_t^y= x-y + \int_0^t \left[ \widetilde b_{u,\varepsilon}(\widetilde X_s^x)- \widetilde b_{u,\varepsilon}(\widetilde X_s^y) \right]ds + \int_0^t \left[ \widetilde\sigma_{u,\varepsilon}(\widetilde X_s^x) -\widetilde\sigma_{u,\varepsilon}(\widetilde X_s^y) \right]dW_s.
\end{align*}
Using \eqref{eq:binomial-2} along with the Cauchy--Schwarz inequality, the BDG inequality, and the common global Lipschitz constant $L$, we obtain for $0\leq t\leq T$
\[
f(t) \leq 3|x-y|^2 + C\int_0^tf(s)ds,
\]
where
\[
f(t) = \mathbb{E}\left[\sup_{0\leq r\leq t}|\widetilde X_r^x-\widetilde X_r^y|^2 \right].
\]
Here one may take $C=3L^2(T+C_{\mathrm{BDG}})$, independent of $x$ and $y$; the moment bound above ensures that $f$ is finite on $[0,T]$. Finally, Gronwall's inequality yields
\[
f(T) \leq 3e^{CT}|x-y|^2
\]
where $C_T=3e^{CT}$.
\end{proof}

\subsection{Infinitesimal generator}

\begin{theorem}\label{thm:companion-generator}
Under Assumption \ref{ass:classical-framework}, let $X_{t\wedge\tau_\Omega}$ denote the solution of \eqref{eq:intro-profile-sde} stopped at the boundary, including its finite exit point.  Its interior infinitesimal generator, understood as a pointwise small-time limit at $x\in\Omega$, is
\begin{equation}\label{eq:companion-L}
\mathcal{L}_{u,\varepsilon}\phi(x)=b_{u,\varepsilon}(x)\cdot\nabla\phi(x)+D_{u,\varepsilon}(x)\fro D^2\phi(x), \quad \phi\in C^2(\mathcal{U}),\quad x\in\Omega.
\end{equation}
The boundary-stopped process is constant after exit; the given differential expression describes its motion in the interior.  In particular,
\begin{equation}\label{eq:profile-reproduction}
\mathcal{L}_{u,\varepsilon}u=\mathrm{div} A_\varepsilon(x,\nabla u)\,\text{ in }\Omega.
\end{equation}
\end{theorem}
\begin{proof} We prove the result in two main parts. First, we derive the infinitesimal generator of the regularized profile-dependent It\^o diffusion. Second, we apply that generator to the same profile \(u\) used to define the drift and covariance coefficients. \\
For notational convenience, throughout the proof, we  write
\[
b(x):=b_{u,\varepsilon}(x), \quad \sigma(x):=\sigma_{u,\varepsilon}(x), \quad a(x):=\sigma(x)\sigma(x)^T.
\]

Fix $x\in\Omega$ and initially let $\phi\in C^2(\Omega)$. Choose an open ball $H$ with $x\in H$ and $K_H:=\overline H\Subset\Omega$. We stop on leaving the open set $H$
\begin{equation} \label{eq:thm74-compact-exit-time}
\tau_H=\inf\{t\geq 0:X_t\notin H\}.
\end{equation}
Continuity gives $\tau_H>0$ almost surely and $X_{t\wedge\tau_H}\in K_H$. The coefficients and the derivatives of $\phi$ are bounded on $K_H$. We may therefore apply It\^o's formula up to $t\wedge\tau_H$.  At the end, for the theorem's test functions $\phi\in C^2(\mathcal{U})$, the same argument applies with $H$ replaced by $\Omega$ because $\overline\Omega\Subset\mathcal{U}$.\\
Let $X_t = \begin{pmatrix} X_t^1 \dots X_t^N\end{pmatrix}^T$, and let $W_t$ be an $m$-dimensional Brownian motion $W_t = \begin{pmatrix} W_t^1 \dots W_t^m \end{pmatrix}^T$.
For this componentwise calculation one may allow an $N\times m$ diffusion matrix; in the present construction $m=N$. The following differential identities are used before $\tau_H$, and their integrated stopped forms are used below. The componentwise SDE is
\begin{equation}\label{eq:thm74-component-SDE}
dX_t^i = b_i(X_t) dt + \sum_{k=1}^m \sigma_{ik}(X_t) dW_t^k, \quad i=1,\ldots,N. 
\end{equation}
Here $b_i$ is the $i$-th component of the drift vector, and $\sigma_{ik}$ is the $(i,k)$-entry of the diffusion matrix.\\
For a $C^2$ function $\phi:\Omega\to\mathbb R$,  the multidimensional It\^o formula gives
\begin{align}\label{eq:thm74-Ito-fully-expanded}
d\phi(X_t) = \left[ \sum_{i=1}^N b_i(X_t)\phi_{x_i}(X_t) \right. \left. + \frac12 \sum_{i,j=1}^N a_{ij}(X_t)\phi_{x_ix_j}(X_t) \right] dt
+ \sum_{i=1}^N \sum_{k=1}^m \phi_{x_i}(X_t) \sigma_{ik}(X_t) dW_t^k. 
\end{align}
In vector--matrix notation \eqref{eq:thm74-Ito-fully-expanded} becomes 
\begin{equation} \label{eq:thm74-Ito-matrix-form}
d\phi(X_t) = \left[ b(X_t)\cdot\nabla\phi(X_t) + \frac12 a(X_t):D^2\phi(X_t) \right]dt + \nabla\phi(X_t)^T \sigma(X_t)dW_t,
\end{equation}
where $D^2\phi = \left( \phi_{x_ix_j} \right)_{i,j=1}^N$ is the Hessian, and $a:D^2\phi=\sum_{i,j=1}^N a_{ij}\phi_{x_ix_j}$.\\
The infinitesimal generator \(\mathcal L\) of  $X_t$ is the differential operator multiplying $dt$ in It\^o's formula \eqref{eq:thm74-Ito-matrix-form}; that is,
\begin{equation}\label{eq:thm74-general-generator}
\mathcal L\phi(x) = b(x)\cdot\nabla\phi(x) + \frac12 a(x):D^2\phi(x). 
\end{equation}
We now insert the profile-dependent coefficients.\\
By construction,
\[
a_{u,\varepsilon}(x) = \sigma_{u,\varepsilon}(x) \sigma_{u,\varepsilon}(x)^T = 2D_\varepsilon(x,\nabla u(x)).
\]
Therefore,
\begin{align}\label{eq:thm74-factor-two-cancellation}
\frac12 a_{u,\varepsilon}(x):D^2\phi(x) = \frac12 \left[ 2D_\varepsilon(x,\nabla u(x)) \right]:D^2\phi(x) = D_\varepsilon(x,\nabla u(x)):D^2\phi(x). 
\end{align}
The factor $2$ in the covariance cancels the factor $1/2$ in the It\^o generator. Using these in formula \eqref{eq:thm74-general-generator} gives
\begin{equation}\label{eq:thm74-profile-generator}
\mathcal L_{u,\varepsilon}\phi(x) = b_{u,\varepsilon}(x)\cdot\nabla\phi(x) + D_\varepsilon(x,\nabla u(x)):D^2\phi(x).
\end{equation}
This identifies the candidate generator, the pointwise small-time limit is justified next.\\
For the boundary-stopped process $X_{t\wedge\tau_\Omega}$ and $\phi\in C^2(\mathcal{U})$, the pointwise interior generator is the limit
\begin{equation}\label{eq:thm74-small-time-generator}
\mathcal L_{u,\varepsilon}\phi(x)=\lim_{t\downarrow0}\frac{\mathbb{E}^x[\phi(X_{t\wedge\tau_\Omega})]-\phi(x)}{t},\quad x\in\Omega.
\end{equation}
We first prove the localized version of \eqref{eq:thm74-small-time-generator} for $H\Subset\Omega$ and $\phi\in C^2(\Omega)$. Thus integrating \eqref{eq:thm74-Ito-matrix-form} up to $t\wedge\tau_H$ gives
\begin{align}\label{eq:thm74-integrated-Ito}
\phi(X_{t\wedge\tau_H})-\phi(x)=\int_0^{t\wedge\tau_H} \left[b(X_s)\cdot\nabla\phi(X_s) + \frac12a(X_s):D^2\phi(X_s)\right]ds+\int_0^{t\wedge\tau_H}\nabla\phi(X_s)^\top\sigma(X_s)dW_s.
\end{align}
Because
\[
\mathbb{E}^x\int_0^{t\wedge\tau_H}|\sigma(X_s)^\top\nabla\phi(X_s)|^2ds\leq t\sup_{y\in K_H}|\sigma(y)^\top\nabla\phi(y)|^2<\infty,
\]
the stochastic integral is square integrable on every finite time interval, and hence, its expectation is zero. Thus, taking expectations and then dividing by $t>0$ gives
\begin{align}\label{eq:thm74-generator-average}
\frac{\mathbb{E}^x[\phi(X_{t\wedge\tau_H})]-\phi(x) }{t} &=\mathbb{E}^x \left[ \frac1t \int_0^{t\wedge\tau_H} \mathcal L_{u,\varepsilon}\phi(X_s)ds \right]. 
\end{align}
As $t\downarrow0$, continuity of $X_s$ and $X_0=x$ imply $X_s\rightarrow x$ almost surely. Since $\mathcal L_{u,\varepsilon}\phi$ is continuous, it follows that $\mathcal L_{u,\varepsilon}\phi(X_s) \rightarrow \mathcal L_{u,\varepsilon}\phi(x)$. As well, the absolute value of the integrand is bounded by $\sup_{K_H}|\mathcal L_{u,\varepsilon}\phi|$, independently of $t$.
Thus the Dominated Convergence Theorem proves the localized generator limit of \eqref{eq:thm74-small-time-generator}. For $\phi\in C^2(\mathcal{U})$, the coefficients, $\phi$, and its first two
derivatives are bounded on $\overline\Omega$.  Thus the preceding It\^o and dominated-convergence argument applies directly with $H=\Omega$, proving \eqref{eq:thm74-small-time-generator}.  This
establishes the theorem's first assertion; that is, the operator in \eqref{eq:thm74-profile-generator} is indeed the infinitesimal generator of the diffusion.\\
Now, we shall apply the generator to the profile $u$.
Set $\phi=u$. Then formula \eqref{eq:thm74-profile-generator} with $b_{u,\varepsilon}(x) = c_\varepsilon(x,\nabla u(x))$ gives 
\begin{align}\label{eq:thm74-generator-on-profile}
\mathcal L_{u,\varepsilon}u = c_\varepsilon(x,\nabla u)\cdot\nabla u + D_\varepsilon(x,\nabla u):D^2u. 
\end{align}
The right-hand side of \eqref{eq:thm74-generator-on-profile} is exactly the right-hand side of \eqref{eq:reg-expansion}. Therefore,
\[
\mathcal L_{u,\varepsilon}u = \operatorname{div} A_\varepsilon(x,\nabla u)
\]
which is formula \eqref{eq:profile-reproduction}. Writing $r=r_\varepsilon(\nabla u)$, the identity is
\begin{equation}\label{eq:profile-reproduction-aa}
 \mathcal L_{u,\varepsilon}u=r^{p(x)-2}\Delta u+(p(x)-2)r^{p(x)-4}\nabla u^\top D^2u\,\nabla u+r^{p(x)-2}\log r\,\nabla p\cdot\nabla u=\mathrm{div} A_\varepsilon(x,\nabla u).
\end{equation}
All terms remain defined at $\nabla u=0$ because $r\geq\varepsilon>0$. This completes the proof.
\end{proof}
\begin{remark}[\textit{Frozen-profile generator versus Fr\'echet linearization}]\label{rem:frechet-linearization} 
We would like to note that the profile-dependent stochastic generator \(\mathcal L_{u,\varepsilon}\) and the full Fr\'echet linearization $\mathcal N_\varepsilon'(u)$ are generally different operators. 
For a precise Fr\'echet statement, fix a ball $B\Subset\mathcal U$
and consider the map
\[
\mathcal N_\varepsilon:C^2(\overline B)\longrightarrow C(\overline B),
\quad
\mathcal N_\varepsilon(w)=\operatorname{div}A_\varepsilon(x,\nabla w).
\]
Its Fr\'echet derivative at the restricted profile $u$ is
\begin{equation}\label{eq:frechet-linearization}
\mathcal N_\varepsilon'(u)[v]=\operatorname{div}\left(D_\varepsilon(x,\nabla u)\nabla v\right),\quad v\in C^2(\overline B).
\end{equation}
For a matrix field $D$, using the column-divergence convention given in \eqref{eq:framework-matrix-divergence}, and expanding \eqref{eq:frechet-linearization} gives
\[
\mathcal N_\varepsilon'(u)[v] = D_\varepsilon(x,\nabla u):D^2v + \operatorname{div} \bigl(D_\varepsilon(x,\nabla u)\bigr) \cdot\nabla v.
\]
By contrast, the stochastic generator is
\[
\mathcal L_{u,\varepsilon}v = D_\varepsilon(x,\nabla u):D^2v + c_\varepsilon(x,\nabla u)\cdot\nabla v.
\]
The total spatial divergence of the composed matrix generally differs from $c_\varepsilon(x,\nabla u)$ which explains the distinction between the two operators. The defining property of the stochastic generator is instead
\[
\mathcal L_{u,\varepsilon}u = \mathcal N_\varepsilon(u).
\]
Therefore, $\mathcal L_{u,\varepsilon}$ should be understood as the profile-frozen nondivergence operator whose action on the defining profile $u$ reproduces the regularized nonlinear divergence operator \eqref{eq:intro-main-operator}.
\end{remark}

\subsection{Consequences and stochastic identities}

\begin{corollary}\label{cor:generator-eigenprofile}
Under Assumption \ref{ass:classical-framework}, suppose
\[
 -\mathrm{div} A_\varepsilon(x,\nabla u)=\lambda\abs{u}^{p(x)-2}u\,\text{ in }\Omega.
\]
Here $\lambda\in\mathbb R$, and $|s|^{ p(x)-2}s$ is assigned the value zero at $s=0$, its continuous extension since $p(x)>1$. Then
\[
 \mathcal{L}_{u,\varepsilon}u=-\lambda\abs{u}^{p(x)-2}u.
\]
\end{corollary}
\begin{proof}
By Theorem \ref{thm:companion-generator},
\[
\mathcal L_{u,\varepsilon}u = \mathrm{div} A_\varepsilon(x,\nabla u).
\]
The eigenprofile equation gives
\[
\mathrm{div} A_\varepsilon(x,\nabla u) = -\lambda |u|^{p(x)-2}u.
\]
Combining the two identities yields
\[
\mathcal L_{u,\varepsilon}u = -\lambda |u|^{p(x)-2}u.
\]
\end{proof}
\begin{corollary}\label{cor:local-martingale}
Suppose Assumption \ref{ass:classical-framework} holds. Let $\phi\in C^2(\Omega)$, let $H\Subset\Omega$ be open, and take $X_0=x\in H$.  Write $\tau_H=\inf\{t\geq0:X_t\notin H\}$.
Then
\[
 M_{t\wedge\tau_H}^\phi=\phi(X_{t\wedge\tau_H})-\phi(x)-\int_0^{t\wedge\tau_H}\mathcal{L}_{u,\varepsilon}\phi(X_s)ds
\]
is a square-integrable martingale on every finite time interval.
\end{corollary}
\begin{proof}
It\^o's formula gives
\[
 M_{t\wedge\tau_H}^\phi=\int_0^t\mathbf1_{\{s<\tau_H\}}\nabla\phi(X_{s\wedge\tau_H})^\top\sigma_{u,\varepsilon}(X_{s\wedge\tau_H})dW_s.
\]
Both $\nabla\phi$ and $\sigma_{u,\varepsilon}$ are bounded on $\overline H\Subset\Omega$, so for every $T<\infty$,
\[
\mathbb{E}^x\int_0^{T\wedge\tau_H}|\sigma_{u,\varepsilon}(X_s)^\top\nabla\phi(X_s)|^2ds\leq T\sup_{y\in\overline H}|\sigma_{u,\varepsilon}(y)^\top\nabla\phi(y)|^2<\infty.
\]
Thus the process is a true square-integrable martingale on $[0,T]$.
\end{proof}
Lastly, we present a result regarding Dynkin's formula.
\begin{corollary}\label{cor:dynkin}
Under Assumption \ref{ass:classical-framework}, let $\theta$ be an $(\mathcal F_t)$-stopping time satisfying $0\leq\theta\leq T\wedge\tau_\Omega$ for some deterministic $T<\infty$.  For every $\phi\in C^2(\mathcal{U})$,
\begin{equation}\label{eq:dynkin}
\mathbb{E}^x\phi(X_\theta)=\phi(x)+\mathbb{E}^x\int_0^\theta\mathcal{L}_{u,\varepsilon}\phi(X_s)ds.
\end{equation}
\end{corollary}

\begin{proof}
It\^o's formula at the stopping time $\theta$ gives
\begin{equation}\label{eq:dynkin-a}
\phi(X_\theta)=\phi(x)+\int_0^\theta\mathcal{L}_{u,\varepsilon}\phi(X_s)ds+\int_0^\theta\nabla\phi(X_s)^\top\sigma_{u,\varepsilon}(X_s)dW_s.
\end{equation}
Since $\overline\Omega\Subset\mathcal{U}$, all derivatives needed here are bounded on $\overline\Omega$, and
\[
\mathbb{E}^x\int_0^\theta|\sigma_{u,\varepsilon}(X_s)^\top\nabla\phi(X_s)|^2ds\leq T\sup_{y\in\overline\Omega}|\sigma_{u,\varepsilon}(y)^\top\nabla\phi(y)|^2<\infty.
\]
The stopped stochastic integral is the terminal value at $T$ of a square-integrable martingale starting at zero, so its expectation is zero. Taking expectations in \eqref{eq:dynkin-a} proves
\eqref{eq:dynkin}.
\end{proof}

\section{Appendix}\label{app:example}
\subsection*{A nonconstant exponent example}

\begin{example}\label{app:example:statement}
Let $\Omega=(-1,1)$, $\mathcal U=(-2,2)$, $V=(-3/2,3/2)$, and fix $\varepsilon>0$. Take $p(x)=3+x/4$ and $u(x)=x^2/2$, and set
\[
r(x)=\sqrt{\varepsilon^2+x^2},\quad F(x)=r(x)^{p(x)-2}.
\]
The data are smooth on $\mathcal U$ and satisfy $5/2<p(x)<7/2$, so Assumption \ref{ass:classical-framework} holds. Abbreviating the frozen coefficients by $D$, $b$, and $\sigma$,
the flux-based prescriptions give
\begin{align*}
D(x)&=F(x)\left(1+(p(x)-2)\frac{x^2}{r(x)^2}\right),\\
b(x)&=\frac14F(x)\log r(x),\quad \sigma(x)=\sqrt{2D(x)}.
\end{align*}
Since $r\geq \varepsilon$ and $p>2$, we have $D>0$. The coefficients are smooth on $\mathcal U$, hence bounded and Lipschitz on $\overline V$, with
$\inf_{\overline V}D>0$. Theorem \ref{thm:sde-existence} therefore yields a pathwise-unique strong solution, with no explosion before exit, of
\[
dX_t=b(X_t)\,dt+\sigma(X_t)\,dW_t,\quad
X_0=x_0\in\Omega,\quad t<\tau_\Omega.
\]
Its interior generator is
\[
\mathcal L\phi=D\phi''+b\phi'.
\]
Because $A_\varepsilon(x,u'(x))=xF(x)$, direct differentiation gives
\[
\mathcal Lu(x)=D(x)+xb(x)=\frac{d}{dx}[xF(x)]=\Delta_{p(\cdot),\varepsilon}u(x).
\]
At the critical point $u'(0)=0$,
\[
D(0)=\varepsilon,\quad b(0)=\frac{\varepsilon\log\varepsilon}{4},\quad \sigma(0)=\sqrt{2\varepsilon}.
\]
Thus $\mathcal Lu(0)=\varepsilon$: regularization preserves ellipticity at zero gradient, while spatial variation of $p$ produces the logarithmic drift. The drift itself is nonzero at
$0$ when $\varepsilon\neq 1$, although $b(0)u'(0)=0$.
\end{example}

\begin{remark}\label{app:example:nonuniqueness}
We would like to note that reproduction alone does not select the coefficients. For $\kappa\geq0$, the smooth fields
$\widehat D=(1+\kappa x^2)D$ and
$\widehat b=b-\kappa xD$ satisfy
$\widehat D+x\widehat b=D+xb$ and $\widehat D>0$.
For $\kappa>0$ they differ from the displayed coefficients away from $0$, illustrating why the flux-based selection rules matter.
\end{remark}

\section*{Funding}
This work was supported by Athabasca University Research Incentive Account [140111 RIA].

\end{document}